\documentclass[a4paper,12pt]{article}

\usepackage[margin=3cm]{geometry}

\usepackage{amsmath, amssymb, amsthm}
\usepackage{color}

\usepackage[unicode]{hyperref}
\usepackage[T1]{fontenc}
\usepackage[utf8]{inputenc}

\usepackage{authblk}

\usepackage{tikz}
\usepackage{stackengine}
\usepackage{graphicx}
\usepackage{mathabx}
\usepackage{microtype}

\DeclareMathOperator{\tw}{tw}%tw_C(n):=max tw of permutations of size n in C
\DeclareMathOperator{\Av}{Av}
\newcommand{\st}[2]{\mathrm{st}(#1;#2)}%sorting time
\newcommand{\wst}[2]{\mathrm{wst}(#1;#2)}%worst-case sorting time

\newcommand{\cC}{\mathcal{C}}
\newcommand{\cD}{\mathcal{D}}
\newcommand{\cS}{\mathcal{S}} %all permutations
\newcommand{\bbN}{\mathbb{N}}

\newcommand{\cA}{\mathcal{A}}
\newcommand{\cB}{\mathcal{B}}
\newcommand{\cL}{\mathcal{L}}
\newcommand{\cF}{\mathcal{F}}
\newcommand{\PBT}{\mathcal{PBT}}
\newcommand{\cR}{\mathcal{R}}
\newcommand{\cRR}{\mathcal{RR}}
\newcommand{\cT}{\mathcal{T}}
\newcommand{\Fringe}[1]{#1\text{-}Fringe}
\newcommand{\RFringe}[1]{#1\text{-}RFringe}
\newcommand{\cCpan}{\cC_{\normalfont\textsf{Pan}}}%pancake sort
\newcommand{\cCins}{\cC_{\normalfont\textsf{Ins}}}%insertion sort
\newcommand{\cCbub}{\cC_{\normalfont\textsf{Bub}}}%bubble sort

\DeclareMathOperator{\Grid}{Grid}%grid class
\DeclareMathOperator{\rin}{rin}%reduced inversion number
\DeclareMathOperator{\cd}{cd}%cyclic distance
\DeclareMathOperator{\tcd}{tcd}%total cyclic distance
\DeclareMathOperator{\Int}{Int}%intervalicity

\newcommand{\Inc}{\raisebox{-0.3mm}{\includegraphics[height=0.7em]{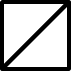}}}
\newcommand{\Dec}{\raisebox{-0.3mm}{\includegraphics[height=0.7em]{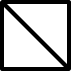}}}
\newcommand{\one}{\bullet}

\newtheorem{problem}{Problem}[section]

\newtheorem{theorem}{Theorem}[section]
\newtheorem{fact}[theorem]{Fact}
\newtheorem{corollary}[theorem]{Corollary}
\newtheorem{lemma}[theorem]{Lemma}

\newtheorem{observation}[theorem]{Observation}
\newtheorem{proposition}[theorem]{Proposition}

\theoremstyle{definition}

\theoremstyle{remark}

\begin{document}

\title{The Hierarchy of Hereditary Sorting Operators}
\author[1]{Vít Jelínek\thanks{Supported by the Czech Science Foundation under the grant agreement no. 23-04949X.}}
\author[2]{Michal Opler\thanks{Supported by the Czech Science Foundation Grant no. 22-19557S.}}
\author[1]{Jakub Pekárek}

\affil[1]{Charles University, Czechia} 
\affil[2]{Czech Technical University in Prague, Czechia}
%\author{Vít Jelínek\thanks{Computer Science Institute, Charles University, Prague. Supported by the Czech Science Foundation under the grant agreement no. 23-04949X.}
%\and Michal Opler\thanks{Czech Technical University, Prague. Supported by the Czech Science Foundation Grant no. 22-19557S.}
%\and Jakub Pekárek\thanks{Computer Science Institute, Charles University, Prague}}

\date{}

\maketitle

\begin{abstract} \small\baselineskip=9pt 
We consider the following general model of a sorting procedure: we fix a hereditary permutation 
class $\cC$, which corresponds to the operations that the procedure is allowed to perform in a 
single step. The input of sorting is a permutation $\pi$ of the set $[n]=\{1,2,\dotsc,n\}$, i.e., a 
sequence where each element of $[n]$ appears once. In every step, the sorting procedure picks a 
permutation $\sigma$ of length $n$ from $\cC$, and rearranges the current permutation of numbers by 
composing it with $\sigma$. The goal is to transform the input $\pi$ into the sorted sequence 
$1,2,\dotsc,n$ in as few steps as possible.

Formally, for a hereditary permutation class $\cC$ and a permutation $\pi$ of $[n]$, we say that 
\emph{$\cC$ can sort $\pi$ in $k$ steps}, if the inverse of $\pi$ can be obtained by composing $k$ 
(not necessarily distinct) permutations from~$\cC$. The \emph{$\cC$-sorting time of $\pi$}, denoted 
$\st\cC\pi$, is the smallest $k$ such that $\cC$ can sort $\pi$ in $k$ steps; if no such $k$ 
exists, we put $\st\cC\pi=+\infty$. For an integer $n$, the \emph{worst-case $\cC$-sorting 
time}, denoted $\wst\cC n$, is the maximum of $\st\cC\pi$ over all permutations $\pi$ of $[n]$.

This model of sorting captures not only classical sorting algorithms, like insertion sort or bubble 
sort, but also sorting by series of devices, like stacks or parallel queues, as well as sorting by 
block operations commonly considered, e.g., in the context of genome rearrangement.

Our goal is to describe the possible asymptotic behavior of the function $\wst\cC n$, and relate 
it to structural properties of $\cC$. As the main result, we show that any hereditary permutation 
class $\cC$ falls into one of the following five categories:

\begin{itemize}
\item $\wst\cC n=+\infty$ for $n$ large enough,
\item $\wst\cC n=\Theta(n^2)$,
\item $\Omega(\sqrt{n}) \le \wst\cC n\le O(n)$,
\item $\Omega(\log n)\le \wst\cC n \le O(\log^2 n)$, or
\item $\wst\cC n=1$ for all $n\ge 2$.
\end{itemize}

In addition, we characterize the permutation classes in each of the five categories.
\end{abstract}

\section{Introduction}\label{sec-intro}

Sorting is, needless to stress, a central concept in computer science. The aim of this paper is to 
study a general way to formalize the notion of sorting, based upon the concept of 
hereditary permutation classes. In this approach, we fix a hereditary class $\cC$ of 
permutations, which corresponds to the operations that we are allowed to perform in a single step 
of sorting. Given an input sequence of $n$ distinct values, itself viewed as a permutation, we aim
to sort it into an increasing sequence by a sequence of steps, where in each step we compose the 
current sequence with an element of~$\cC$. The goal is to minimize the number of steps for a given 
input, and to determine the worst-case number of steps $\cC$ needs on inputs of size~$n$. This 
last quantity, which we call the \emph{worst-case sorting time} of the class $\cC$, is the main 
focus of this paper.

Our formalism is general enough to encompass standard sorting procedures, like sorting by adjacent 
transpositions (``bubble-sort'') or sorting by insertions (``insertion sort''), as well as parallel 
procedures like Habermann's ``parallel neighbor sort''~\cite{Habermann} (a.k.a. ``odd-even sort''), 
which allows performing any number of adjacent swaps in a single step. 

Our formalism also models sorting procedures which are described via operations not on single 
elements, but on blocks of consecutive elements, e.g. block reversals or block transpositions. Such 
block-based sorting occurs, e.g., in the well-known toy problem known as ``pancake 
sorting''~\cite{BFR,Gates}, where the goal is to sort the input by a sequence of prefix reversals. 
Block operations are also the backbone of the sorting procedures studied in the context of genome 
rearrangement research (see, e.g.,~\cite{Bafna,CF,HomVat,Kececioglu}, or the surveys 
in~\cite{Bhatia,Fertin}).

Finally, we point out that our sorting formalism is closely related to the concept of sorting 
networks, an approach which dates back to the works of Knuth~\cite{Knuth68} and 
Tarjan~\cite{Tarjan} in late 1960s and early 1970s, respectively.
A sorting network is an acyclic directed graph whose every vertex represents a 
``sorting device'' (e.g., a stack, a queue,  a double-ended queue, or a more general 
``$\cC$-container'' considered by Albert et al.~\cite{containers}). An input sequence enters into 
the network from a designated source, its elements can move along the edges and be rearranged by 
the devices at the vertices, with the goal that the elements reach a designated sink in
ascending order. Since the action of most sorting devices considered in the literature can be 
modelled by a hereditary permutation classes, our sorting formalism can be seen as a serial sorting 
network (i.e., the network is a path) in which all the devices are of the same type.

The topic of sorting networks continues to be a major focus of research~\cite{Bona, Cerbai2, 
Cerbai1,EPG}, but we note that most of its previous results address the question of how many 
permutations can be sorted by a given fixed network and what these permutations look like, while 
our own focus is on finding the shortest serial network that sorts every permutation of a given 
length. Additionally, some authors consider sorting networks where the operation of the devices is 
restricted by the properties of the input sequence, e.g., by assuming that a stack may only hold a 
monotone sequence of values. These restricted devices no longer fit into our own formalism.

The aim of our paper is neither to improve the analysis of any specific sorting algorithm, nor to 
propose a new specific algorithm; instead, our goal is to answer the general question of which 
functions can occur as worst-case sorting times of hereditary permutation classes. Accordingly, 
we disregard constant multiplicative factors and focus on asymptotic behavior. Our results show  
that at this level of detail, the possible sorting times can be classified into a hierarchy with 
gaps between individual levels. Disregarding trivial extreme cases, a worst-case sorting time of a 
hereditary class is either $\Theta(n^2)$, or a function between $O(n)$ and $\Omega(\sqrt{n})$, or 
$O(\log^2 n)$, and for each of these cases we can provide a structural characterization of the 
corresponding hereditary classes. 

We remark that the assumption of heredity is substantial in our setting. If we were to consider 
sorting with an arbitrary set of generators, the sorting-time problem would be essentially reduced 
to the problem of determining the diameter of a Cayley graph in the symmetric group. This is a 
well-known group-theoretic open problem, with a conjecture of Babai and Seress~\cite{Babai} stating 
that the diameter is at most polynomial.

Let us now introduce the necessary definitions to state our results properly.

\subsection*{Notation, terminology, and the main result}

A \emph{permutation} of size $n$ is a sequence $\pi=(\pi(1),\pi(2),\dotsc,\pi(n))$ in which each 
number from the set $[n]=\{1,2,\dotsc,n\}$ appears exactly once. We let $\cS_n$ be the set of 
permutations of size $n$, $\cS_{\le n}$ the set of permutations of size at most $n$, and $\cS$ 
the set $\bigcup_{n\ge 1} \cS_n$.  We say that a 
permutation $\pi\in\cS_n$ \emph{contains} a permutation $\sigma\in\cS_k$ if $\pi$ contains a 
subsequence of length $k$ whose elements are in the same relative order as the elements of $\sigma$, 
or more formally, if there are indices $1\le i_1<i_2<\dotsb<i_k\le n$ such that $\pi(i_a)<\pi(i_b) 
\iff \sigma(a) <\sigma(b)$ for every $a,b\in[k]$. If $\pi$ does not contain $\sigma$, we say it 
\emph{avoids}~$\sigma$.

A \emph{hereditary permutation class} (or henceforth just \emph{class}) is a set $\cC\subseteq\cS$ 
such that for every $\pi\in\cC$ and every $\sigma$ contained in $\pi$, $\sigma$ is also in~$\cC$. 
For a permutation class $\cC$, we let $\cC_n$ denote the set $\cC\cap \cS_n$. 

For a pair of permutations $\sigma,\pi\in\cS_n$, their \emph{composition} $\sigma\circ\pi$ is the 
permutation $(\sigma(\pi(1)),\sigma(\pi(2)),\dotsc,\sigma(\pi(n)))$. For a pair of permutation 
classes $\cA$ and $\cB$, we write $\cA\circ\cB$ for $\{\alpha\circ\beta;\; \alpha\in\cA, 
\beta\in\cB\}$, $\cA^{\circ k}$ for the $k$-fold composition $\cA\circ\cA\circ\dotsb\circ\cA$, and 
$\cA^{*}$ for $\bigcup_{k\ge 1} \cA^{\circ k}$. Note that $\cA\circ\cB$, $\cA^{\circ k}$ and 
$\cA^{*}$ are again permutation classes. 

The \emph{inverse} of a permutation $\pi\in\cS_n$ is the permutation $\pi^{-1}\in\cS_n$ such that 
$\pi^{-1}\circ\pi$ is the identity permutation $(1,2,\dotsc,n)$.

Recall that the \emph{$\cC$-sorting time of $\pi$}, denoted $\st\cC\pi$, is the smallest $k$ such 
that $\pi^{-1}$ is in $\cC^{\circ k}$, or $+\infty$ if no such $k$ exists. For an integer $n$, the 
\emph{worst-case $\cC$-sorting time}, denoted $\wst\cC n$, is the maximum of $\st\cC\pi$ over all 
permutations $\pi$ of $[n]$.

We are now ready to recall the main result of this paper.
\begin{theorem}\label{thm-main}
Every permutation class $\cC$  falls into one of the following five categories:
\begin{itemize}
\item $\wst\cC n=+\infty$ for $n$ large enough,
\item $\wst\cC n=\Theta(n^2)$,
\item $\Omega(\sqrt{n}) \le \wst\cC n\le O(n)$,
\item $\Omega(\log n)\le \wst\cC n \le O(\log^2 n)$, or
\item $\wst\cC n=1$ for all $n\ge 2$.
\end{itemize}
\end{theorem}

The rest of this paper is devoted to the proof of this theorem. After we introduce some more 
terminology and review some relevant facts about permutation classes (Section~\ref{sec-classes}), we 
will proceed by characterizing the classes with infinite sorting time in Section~\ref{sec-inf}. This 
is easily done, as it more or less directly follows from previous results on composition-closed 
permutation classes. We then show, in Section~\ref{sec-kva}, that if the sorting time is finite, it 
is at most $O(n^2)$. Next, in Section~\ref{sec-lin}, we establish the gap between quadratic and 
linear sorting times. Our main device is a numerical parameter called \emph{reduced inversion 
number} $\rin(\pi)$, which we introduce. We show that a class $\cC$ has  $\wst{\cC}{n}=\Omega(n^2)$ 
whenever $\rin(\cdot)$ is bounded on $\cC$, and $\wst{\cC}{n}=O(n)$ otherwise. 

The most difficult part of the proof appears in Section~\ref{sec-polylog}, where we show that there 
are twelve (or four, after we account for symmetries) minimal classes with sorting time $O(\log^2 
n)$, while any class that does not contain one of these twelve as a subclass has sorting time 
$\Omega(\sqrt{n})$. The proof of this lower bound combines a new structural result on classes 
avoiding the twelve `polylog' subclasses, an argument showing that sorting with such a class can be 
encoded into a specific graph which has a drawing on a surface of not-too-large genus with 
not-too-many crossings, and a bound on the tree-width of a graph that can be drawn on a surface of a 
given genus with a given number of crossings. 

The short Section~\ref{sec-all} contains the proof that any class not containing all permutations 
has sorting time $\Omega(\log n)$. This is a straightforward counting argument using the 
Marcus--Tardos theorem. We conclude in Section~\ref{sec-open} with open problems, chief of which is 
to close the gaps between $\Omega(\sqrt{n})$ and $O(n)$, and between $\Omega(\log n)$ and $O(\log^2 
n)$ implied by Theorem~\ref{thm-main}.  In fact, for all the classes whose sorting time we can 
establish asymptotically tightly, the sorting time is either $+\infty$, $\Theta(n^2)$, $\Theta(n)$, 
$\Theta(\log n)$ or~1.

\section{Permutation classes}\label{sec-classes}
We let $\Av(\pi^1,\pi^2,\dotsc,\pi^m)$ denote the class of permutations that avoid 
all the permutations $\pi^1,\dotsc,\pi^m$. We also use the special notation $\Inc$ for the class 
$\Av(21)$ of all the increasing permutations, and $\Dec$ for the class $\Av(12)$ of all the 
decreasing permutations. Let $\iota_k$ denote the identity permutation $(1,2,\dotsc,k)$, and let 
$\delta_k$ be the decreasing permutation $(k,k-1,\dotsc,1)$. 

It is often convenient to represent a permutation $\pi\in\cS_n$ by its \emph{diagram}, which is 
the point set $D(\pi)=\{(i,\pi(i); i\in[n]\}$. 

Apart from the inverse $\pi^{-1}$ introduced before, we will need several other operations 
on a permutation $\pi=(\pi(1),\dotsc,\pi(n))$. Its \emph{reverse}, denoted $\pi^r$, is the 
permutation $(\pi(n),\pi(n-1),\dotsc,\pi(1))$, its \emph{complement} $\pi^c$ is the permutation 
$(n+1-\pi(1),n+1-\pi(2),\dotsc,n+1-\pi(n))$. Finally, the \emph{flip} of $\pi$, denoted $\pi^f$, 
is defined as $\pi^f=\left((\pi^r)^{-1}\right)^r$.

In terms of permutation diagrams, the reverse corresponds to reflection over a vertical line, 
the complement corresponds to reflection over a horizontal line, the inverse corresponds to 
reflection over the diagonal line $y=x$, and flip corresponds to reflection over the line $y=n+1-x$. 

For a set of permutations $\cC$, we write $\cC^r=\{\pi^r;\; \pi\in\cC\}$, 
$\cC^c=\{\pi^c;\; \pi\in\cC\}$, $\cC^{-1}=\{\pi^{-1};\; \pi\in\cC\}$ and $\cC^f=\{\pi^f;\; 
\pi\in\cC\}$. Note that if $\cC$ is a class, then so are $\cC^r$, $\cC^c$, $\cC^{-1}$ and $\cC^f$. 

For a pair of permutations $\alpha\in\cS_k$ and $\beta\in\cS_\ell$, their \emph{direct sum} 
$\alpha\oplus\beta$ is the permutation 
$(\alpha(1),\alpha(2),\dotsc,\alpha(k),k+\beta(1),k+\beta(2),\dotsc,k+\beta(\ell))$, and their 
\emph{skew sum} $\alpha\ominus\beta$ is the permutation
$(\ell+\alpha(1),\ell+\alpha(2),\dotsc,\ell+\alpha(k),\beta(1),\beta(2),\dotsc,\beta(\ell))$. We let 
$\bigoplus^m\alpha$ denote the direct sum $\alpha\oplus\alpha\oplus\dotsb\oplus\alpha$ with $m$ 
summands, and $\bigominus^m\alpha$ is defined similarly. For a set of permutations $\cC$, its 
\emph{sum-closure} $\bigoplus\cC$ is the set of all the permutations that can be obtained as a
direct sum of a finite number of (not necessarily distinct) elements of~$\cC$. Similarly, we define 
the \emph{skew-closure} $\bigominus\cC$. If $\cC$ is a class, then so are $\bigoplus\cC$ 
and~$\bigominus\cC$.

Before we proceed, let us collect several simple observations related to sorting times of various 
classes.

\begin{observation}\label{obs-sym}
For any two permutations $\sigma$ and $\pi$ of size $n$, we have 
$(\sigma\circ\pi)^{-1}=\pi^{-1}\circ \sigma^{-1}$ and $(\sigma\circ\pi)^f=\pi^f\circ 
\sigma^f$. Consequently, for any permutation class $\cC$ and any permutation $\tau$, 
$\st\cC\tau=\st{\cC^{-1}}{\tau^{-1}}=\st{\cC^f}{\tau^f}$, and in particular, 
$\wst{\cC}{n}=\wst{\cC^{-1}}{n}=\wst{\cC^f}{n}$.
\end{observation}

\begin{observation}\label{obs-comp}
If $\cA$ and $\cB$ are permutation classes such that $\cA\subseteq\cB^{\circ k}$ and $\pi$ is any 
permutation, then $\st{\cB}{\pi}\le k\cdot\st{\cA}{\pi}$, and hence $\wst{\cB}{n}\le k\cdot 
\wst{\cA}{n}$.
\end{observation}

\begin{observation}\label{obs-r}
Let $\cC$ be a class containing $\Dec$ as a subclass. Then $\cC^r=\cC\circ\Dec\subseteq \cC^{\circ 
2}$, and in particular, by Observation~\ref{obs-comp}, $\wst{\cC}{n}\le 2\wst{\cC^r}{n}$. Similarly, 
$\wst{\cC}{n}\le 2\wst{\cC^c}{n}$.
\end{observation}

\subsection*{Griddings and grid classes}

We say that a set $P$ of points in the plane is in \emph{general position} if no two points of $P$ 
are on the same horizontal or vertical line. We say that two finite point sets $P$ and $Q$ in 
general position are \emph{order-isomorphic}, if there is a bijection from $P$ to $Q$ which 
preserves both the left-to-right and the bottom-to-top order of the points. Note that any finite set 
$P$ of points in general position is order-isomorphic to a permutation diagram of a unique 
permutation $\alpha$; we than say that $P$ \emph{induces a copy of $\alpha$}. 

\begin{figure}
	\centerline{\includegraphics{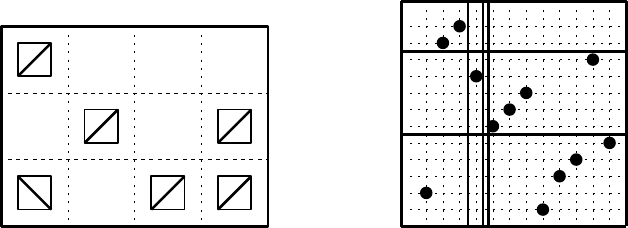}}
	\caption{Left: a $4\times 3$ gridding matrix $M$; by convention, we assume that the empty cells 
		represent the permutation class that only contains the empty permutation. Right: an example of an 
		$M$-gridded permutation, represented by its diagram.}\label{fig-grid}
\end{figure}

Let $\pi\in\cS_n$ be a permutation with a permutation diagram $D(\pi)$. A \emph{$k\times\ell$ 
gridding of $\pi$} is a set of $k+1$ vertical lines $v_0,v_1,\dotsc,v_k$ numbered left to right and 
$\ell+1$ horizontal lines $h_0,h_1,\dotsc,h_\ell$ numbered bottom to top, such that $D(\pi)$ is 
inside the rectangle bounded by $v_0$, $v_k$, $h_0$ and $h_\ell$, and no point of $D(\pi)$ lies on 
any of the $k+\ell$ lines of the gridding. For $i\in[k]$ and $j\in[\ell]$, the \emph{cell $(i,j)$} 
of the gridding, denoted is the rectangle bounded by the lines $v_{i-1}$, $v_i$, $h_{j-1}$ 
and~$h_j$.  A permutation together with its gridding is referred to as a \emph{gridded permutation}. 
If $\pi$ is a gridded permutation and $X$ a cell of its gridding, we let $\pi[X]$ denote the set 
$D(\pi)\cap X$, and the permutation induced by this set is simply referred to as the subpermutation 
of $\pi$ induced by~$X$. 

A \emph{gridding matrix} is a matrix $M$ whose every entry is a permutation class; see 
Figure~\ref{fig-grid} for an example. For consistency with the Cartesian coordinates we use for 
permutation diagrams, we will assume that rows in a matrix are numbered bottom to top, and we let 
$M_{i,j}$ denote the entry in column $i$ and row $j$ of the matrix~$M$. Suppose $M$ is a gridding 
matrix with $k$ columns and $\ell$ rows. An $M$-gridding of a permutation $\pi$ is a $k\times\ell$ 
gridding with the property that for every $(i,j)\in[k]\times[\ell]$, the cell $(i,j)$ of the gridding 
induces in $\pi$ a permutation belonging to the class~$M_{i,j}$. The \emph{grid class} determined by 
$M$, denoted $\Grid(M)$, is the set of those permutations $\pi$ that admit an $M$-gridding.

An important special case of grid classes are the so-called \emph{monotone juxtapositions}, whose 
gridding matrices have shape $2\times 1$ or $1\times 2$, and each of their two entries is $\Inc$ or 
$\Dec$. In particular, there are eight monotone juxtapositions.

\subsection{Some notable classes}
We will now introduce several permutation classes that will play an important role in our results. 
See Figure~\ref{fig-perms} for some examples and Table~\ref{tab-classes} for a complete list.

The class of \emph{layered permutations}, denoted $\cL$, is the sum-closure of the class $\Dec$. 
Equivalently, it can be defined as $\Av(231,312)$. We remark that sorting by layered permutations is 
equivalent by sorting with a series of so-called \emph{pop-stacks}.

The \emph{Fibonacci class}, denoted $\cF$, is the sum-closure of the set $\{1,21\}$, or equivalently 
the class of 321-avoiding layered permutations. Its name refers to the fact that $|\cF_n|$ is the 
$n$-th Fibonacci number. Sorting with the Fibonacci class $\cF$ corresponds to sorting where in each 
step we perform parallel transpositions of disjoint adjacent pairs of elements. This includes, as a 
special case, an algorithm known as odd-even sort (also known as parallel neighbor-sort), introduced 
and analyzed by Habermann~\cite{Habermann}, who showed that the algorithm requires $n$ rounds in 
the worst case. This implies that $\wst{\cF}{n}=O(n)$. This is asymptotically tight since, e.g., any 
permutation $\pi$ with $\pi(n)=1$ requires at least $n-1$ $\cF$-steps to be sorted. Thus, 
$\wst{\cF}{n}=\Theta(n)$.

\begin{figure}
	\centerline{\includegraphics[width=0.9\textwidth]{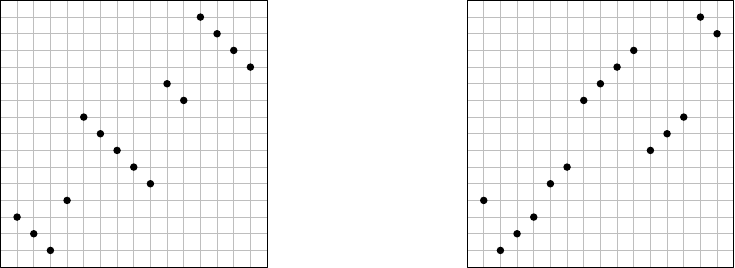}}
	\caption{From left to right: diagram of a layered permutation, diagram of a permutation from the 
		Fibonacci class $\cF$, and a diagram of a permutation from the class~$\PBT$.}\label{fig-perms}
\end{figure}

The \emph{rotation class} is the class $\cR:=\{\iota_a\ominus\iota_b;\; a,b\in\bbN_0\}$. Note that 
for $\rho\in\cR_n$ and $\pi\in \cS_n$, $\pi\circ\rho$ is a cyclic rotation of the sequence~$\pi$. We 
will also consider the class $\cRR:=\cR\cup \cR^r$, corresponding to rotations and their reversals. 
Both these classes are composition closed, and in particular have infinite sorting time, as they 
cannot sort any permutation that does not belong to the class.

The \emph{parallel block transposition class}, denoted $\PBT$, is the sum-closure of~$\cR$. To our 
knowledge, this class has not been considered before, but it emerges naturally in our results as 
one of the twelve minimal classes with polylogarithmic sorting time.

The class $\cCbub:=\{\iota_a\oplus 21\oplus\iota_b;\; a,b\in\bbN_0\}\cup\Inc$ corresponds to sorting 
by adjacent transpositions. This includes, e.g., the classical bubble sort algorithm. It can be 
easily seen that $\wst{\cCbub}{n}=\binom{n}{2}$. A slight modification yields the class
$\cT:=\{\iota_a\oplus 21\oplus\iota_b;\; a,b\in\bbN_0\}\cup\{1\ominus\iota_c\ominus 1;\; 
c\in\bbN_0\}\cup\Inc$, which corresponds to cyclic adjacent transpositions, where positions 1 and 
$n$ are treated as adjacent.

The class $\Fringe k:=\{\pi\oplus\iota_a\oplus\sigma;\; a\in\bbN_0, |\sigma|\le k, |\pi|\le k\}$ 
contains permutations in which only the prefix of length $k$ and the suffix of length~$k$ can be 
permuted. We also define $\RFringe k:=\Fringe{k}\cup(\Fringe{k})^r$.

The grid class 
$\cCpan=\Grid\left(\begin{smallmatrix}&\Inc\\\Dec&\end{smallmatrix}\right)$ 
corresponds to sorting where in each step, we reverse the prefix of the permutation being sorted. 
This is known in the literature as \emph{pancake sorting}. It is well known that
$\wst{\cCpan}{n}=\Theta(n)$. It is easy to see that the same asymptotic estimate also applies to 
$\cCpan^r$, $\cCpan^c$ and $\cCpan^f$.

The grid class 
\[\cCins=\Grid\begin{pmatrix}&&\Inc\\\one&&\\&\Inc&\end{pmatrix}\] 
corresponds to sorting where in each step, we take the first element of the input and insert it to 
an arbitrary position $i$, while moving the elements in positions $2,3,\dotsc,i$ one step to the 
left. This corresponds to the classical insertion sort algorithm. The standard application of the 
algorithm can sort any input in $n-1$ steps, by maintaining the invariant that after $k$ steps, 
the final $k+1$ elements are in increasing order; hence, $\wst{\cCins}{n}\le n-1$. This bound is 
tight, since any permutation $\pi$ with $\pi(n)=1$ requires $n-1$ steps to be sorted.

\begin{table}[t]
\[
\begin{array}[t]{| r |l|}
\hline
\text{Notation} &\text{Description}\\
\hline\hline
\cL^{\vphantom{I}}_{\vphantom{I}} & \bigoplus^{\vphantom{I}}_{\vphantom{I}} \Dec
\\[0pt]\hline
\cF^{\vphantom{I}}_{\vphantom{I}} & \bigoplus^{\vphantom{I}}_{\vphantom{I}} \{1, 12\} = \cL\cap\Av(321)
\\[0pt]\hline
\cR & \Grid\begin{pmatrix}\Inc&\\&\Inc\end{pmatrix}^{\vphantom{I}}_{\vphantom{I}}
\\[0pt]\hline
\cRR& \cR\cup {\cR^r}^{\vphantom{I}}_{\vphantom{I}}
\\[0pt]\hline
\PBT& \bigoplus\cR^{\vphantom{I}}_{\vphantom{I}}
\\[0pt]\hline
\cCbub & 
\Grid\begin{pmatrix}&&&\Inc\\&\one&&\\&&\one&\\\Inc&&&\end{pmatrix}^{\vphantom{I}}_{\vphantom{I}}
\\[0pt]\hline
\end{array}
\qquad\quad
\begin{array}[t]{|r | l|}
\hline
\text{Notation} &\text{Description}\\
\hline\hline
\cT & 
\cCbub\cup\Grid\begin{pmatrix}\one&&\\&\Inc&\\&&\one\end{pmatrix}^{\vphantom{I}}_{\vphantom{I}}
\\[0pt]\hline
\Fringe k & \Grid\begin{pmatrix}&&\cS_{\le k}\\&\Inc&\\\cS_{\le 
k}&&\end{pmatrix}^{\vphantom{I}}_{\vphantom{I}}
\\[0pt]\hline
\RFringe k & \Fringe k\cup {{\Fringe k}^r}^{\vphantom{I}}_{\vphantom{I}}
\\[0pt]\hline
\cCpan & \Grid\begin{pmatrix}&\Inc\\\Dec&\end{pmatrix}^{\vphantom{I}}_{\vphantom{I}}
\\[0pt]\hline
\cCins & \Grid\begin{pmatrix}&&\Inc\\\one&&\\&\Inc&\end{pmatrix}^{\vphantom{I}}_{\vphantom{I}}
\\[0pt]\hline
\end{array}
\]

 \caption{A summary of relevant permutation classes}\label{tab-classes}
\end{table}

\section{Classes that cannot sort}\label{sec-inf}
Let us say that a class $\cC$ \emph{cannot sort}, if $\wst{\cC}{n}=+\infty$ for some~$n$. Note that 
this is equivalent to $\cC^*_n\neq \cS_n$. Moreover, if $\cC^*_n$ is different from $\cS_n$ for 
some 
$n$, then for any $m\ge n$, $\cC^*_m$ is different from $\cS_m$, since $\cC^*$ is a 
permutation class. 
% 
% We will say that a class $\cC$ is a \emph{near-subclass} of a class $\cD$, denoted $\cC\nsub\cD$, 
% if the set $\cC\setminus\cD$ is finite. Note that if $\cD$ cannot sort and $\cC\nsub\cD$, then 
% $\cC$ cannot sort either.

When seeking to characterize the maximal classes that cannot sort, we may restrict our attention to 
the classes $\cC$ that are \emph{composition closed}, that is, those that satisfy $\cC=\cC^*$. 
Composition closed permutation classes have been characterized by Atkinson and Beals~\cite{AB}. 
Translating their results into our terminology, we obtain the following proposition.

\begin{proposition}\label{pro-cannot}
A permutation class $\cC$ cannot sort if and only if it satisfies, for some $k\in\bbN$,  one of the 
following two 
conditions:
\begin{enumerate}
\item $\cC$ is a subclass of $\cRR\cup \cS_{\le k}$, or
\item $\cC$ is a subclass of $\RFringe{k}$.
\end{enumerate}
\end{proposition}
% Note that in the second condition of Proposition~\ref{pro-cannot}, we may replace `subclass' with 
% `near-subclass' without changing the meaning (but possibly changing the value of~$k$).

\section{From infinity to \texorpdfstring{$O(n^2)$}{O(n\texttwosuperior)}}\label{sec-kva}

Let us now focus on showing that any class that can sort has at most quadratic sorting 
time. Our key concept will be the notion of \emph{peg class}, which is a grid class with a gridding 
matrix whose every row and column has exactly one nonempty cell, and each nonempty cell is either 
the singleton class $\{1\}$ or one of the two monotone classes $\Inc$ and~$\Dec$. 

Peg classes have been previously studied mostly in the context of enumeration~\cite{AABRV, HV, 
HomVat}. Indeed, it is known that a permutation class has polynomial growth rate if and only if it 
is a finite union of peg classes. For our purposes, we will use the following result.

\begin{fact}[\protect{Homberger--Vatter~\cite[Theorem 3.1]{HomVat}}]\label{fac-peg}
If a permutation class $\cC$ does not contain as a subclass any monotone juxtaposition, nor 
$\cF$ nor $\cF^r$, then $\cC$ is a finite union of peg classes. 
\end{fact}

% A \emph{concatenation} of permutation classes $\cA$ and $\cB$, denoted $\cA|\cB$, is the set of 
% permutations that can be obtained by concatenating a (possibly empty) sequence order-isomorphic to 
% a permutation from $\cA$ with a (possibly empty) sequence order-isomorphic to a permutation 
% from~$\cB$. A \emph{monotone juxtaposition} is a class of the form $\cA|\cB$ or $(\cA|\cB)^{-1}$ 
% with $\cA,\cB\in\{\Inc,\Dec\}$ (so there are eight monotone juxtapositions in total).

It is not hard to see that any monotone juxtaposition has worst-case sorting time of order $O(\log 
n)$; we later prove this in Proposition~\ref{pro-lin2}. Also for the Fibonacci class $\cF$, we 
have already pointed out that $\wst{\cF}{n}=\Theta(n)$, and we may easily observe that 
$\wst{\cF^r}{n}=O(n)$.  

Thus, if a class $\cC$ has at least quadratic $\wst{\cC}{n}$, it cannot contain any monotone 
juxtaposition and any symmetry of $\cF$ as a subclass. Thus, by Fact~\ref{fac-peg}, such a class 
$\cC$ can be written as a finite union of \emph{peg classes}.

\begin{lemma}\label{lem-peg}
 Let $\cC$ be a peg class. If $\cC$ can sort, then $\wst{\cC}{n}=O(n^2)$.
\end{lemma}
\begin{proof}
If the gridding matrix of $\cC$ contains a cell equal to $\Inc$ as well as a cell equal to $\Dec$, 
then $\cC$ contains $\cCpan=\Dec\oplus\Inc$ or one of its symmetries as a subclass, and with the 
help of Observation~\ref{obs-r} we easily conclude that $\wst{\cC}{n}=O(n)$.

Suppose then that at most one of the two classes $\Dec$ and $\Inc$ occurs in the gridding of 
$\cC$. If $\Dec$ occurs, then by Observation~\ref{obs-r}, $\wst{\cC}{n}\le 2\wst{\cC^r}{n}$, and 
$\cC^r$ is a peg class whose gridding matrix contains $\Inc$. It is therefore enough to
deal with peg classes whose gridding matrix contains $\Inc$ but not~$\Dec$. We assume from now on 
that $\cC$ is such a peg class.

Let $M$ be the gridding matrix of the class $\cC$, and assume that $M$ has been chosen to have the 
smallest possible number of rows and columns. In particular, this means that if $M_{i,j}=\Inc$ 
for some $i$ and $j$, then either the cell $M_{i,j}$ is in the last column or last row of $M$, or 
the cell $M_{i+1,j+1}$ is empty. If not, then we could delete column $i+1$ and row $j+1$ from $M$ 
to obtain a smaller gridding matrix representing the same class~$\cC$. By the same reasoning, if 
$M_{i,j}=\Inc$ for some $i>1$ and $j>1$, then $M_{i-1,j-1}$ is empty.

\begin{figure}
\begin{align*}
&\cC_a=\Grid\begin{pmatrix}
        \one&&\\&\one&\\&&\Inc
       \end{pmatrix},\,
\cC_b=\Grid\begin{pmatrix}
        \one&&\\&\Inc&\\&&\one
       \end{pmatrix},\,
\cC_c=\Grid\begin{pmatrix}
        &&\one\\\one&&\\&\Inc&
       \end{pmatrix}, \\
% \cC_d&=\begin{pmatrix}
%         &\one&\\&&\Inc\\\one&&
%        \end{pmatrix}, &
&\cC_d=\Grid\begin{pmatrix}
        &&&\Inc\\ &\one&&\\ &&\one&\\ \Inc&&&
       \end{pmatrix},\, 
\cC_e=\Grid\begin{pmatrix}
        \one&&&\\ &&&\Inc \\ &\one&& \\ &&\one&
       \end{pmatrix}.
\end{align*}
\caption{Five peg classes that can sort in quadratic time.}\label{fig-pegs}
\end{figure}

Consider now the five peg classes from Figure~\ref{fig-pegs}. We claim that all these classes, and 
therefore also their inverses and flips, have at most quadratic sorting time. 

Consider the class $\cC_a$, which corresponds to sorting where in every step we either 
move the first element of the input to the last position, shifting the remaining elements one step 
forward, or we move the first two elements to the end, while simultaneously changing their relative 
position. We may easily see that by a sequence of at most $n$ steps of $\cC_a$, we may simulate a 
single step of the class $\cF$ of parallel adjacent transpositions on an input of size~$n$. As we 
pointed out in Section~\ref{sec-classes},
$\wst{\cF}{n}=O(n)$. It follows that 
$\wst{\cC_a}{n}=\wst{\cC_a^{-1}}{n}=\wst{\cC_a^{f}}{n}=O(n^2)$. The same argument also applies to 
the remaining classes $\cC_b$, $\cC_c$, $\cC_d$ and~$\cC_e$.

\newcommand{\Mtl}{\ensuremath{M_\mathrm{TL}}}
\newcommand{\Mtr}{\ensuremath{M_\mathrm{TR}}}
\newcommand{\Mbl}{\ensuremath{M_\mathrm{BL}}}
\newcommand{\Mbr}{\ensuremath{M_\mathrm{BR}}}

Consider again the peg class $\cC$ from our lemma, with its gridding matrix~$M$. We will show that
$\cC$ either contains, as a subclass, one of the classes $\cC_a,\dotsc,\cC_e$ from 
Figure~\ref{fig-pegs} or its inverse or flip. We know that $M$ contains at least one entry equal to 
$\Inc$, so let us fix one such entry $M_{i,j}$. We now let $\Mtl$, $\Mtr$, $\Mbl$ and $\Mbr$ denote 
the submatrices of $M$ formed be the cells strictly to the top-left, top-right, bottom-left and 
bottom-right of the cell $M_{i,j}$, respectively; for instance $\Mbl$ is formed by the intersection 
of the leftmost $i-1$ columns and bottommost $j-1$ rows of~$M$.

If both $\Mtl$ and $\Mbr$ contain a non-empty cell, then $\cC$ contains $\cC_b$ as a subclass, so 
suppose this is not the case, and assume, without loss of generality, that $\Mbr$ only contains 
empty cells. If the grid class of $\Mtl$ contains $21$, then $\cC$ contains $\cC_a$ and we are 
done, so we assume from now on that $\Mtl$ generates a (possibly empty) subclass of~$\Inc$. It 
follows that at least one of $\Mbl$ and $\Mtr$ must contain a nonempty cell, otherwise $\cC$ would 
be a subclass of $\cR$ and it could not sort. 

Suppose without loss of generality that $\Mtr$ contains at least as many $\Inc$-cells as $\Mbl$, 
and 
in case neither $\Mtr$ nor $\Mbl$ contains any $\Inc$-cell, suppose that $\Mtr$ has at least 
one nonempty cell. Let $M_{i+1,k}$ be a nonempty cell in the column of $M$ right next 
to~$M_{i,j}$. Note that $k$ is at least $j+2$, since $\Mbr$ only has empty cells and if the cell 
$M_{i+1,j+1}$ were nonempty, we could combine it with $M_{i,j}$ contradicting the minimality 
of~$M$. 

Let $M_{\ell,j+1}$ be a nonempty cell of $M$ in the row above $M_{i,j}$. If $M_{\ell,j+1}$ is 
inside $\Mtl$, then $\cC$ contains $\cC_c$ as a subclass, so suppose that $M_{\ell,j+1}$ is inside 
$\Mtr$. If at least one of the two cells $M_{i+1,k}$ and $M_{\ell,j+1}$ is equal to $\Inc$, then 
$\cC$ contains either $\cC_c^f$ or $(\cC_c^f)^{-1}$ as a subclass and we are done. Suppose that this 
is not the case. Since $\cC$ is not a subclass of $\Fringe{k}$ for any fixed $k$, there must be 
either a nonempty cell in $\Mtl$ or a cell equal to $\Inc$ in~$\Mtr$. In the former case, $\cC$ 
contains as a subclass either $\cC_c$ or $\cC_e^f$. In the latter case, it contains either $\cC_d$ 
or 
a class symmetric to $\cC_c$ as a subclass. 

In all the cases considered, we found that $\cC$ contains a class symmetric to one of 
$\cC_a,\dotsc, \cC_e$ as a subclass, proving that $\wst{\cC}{n}=O(n^2)$ as claimed.
\end{proof}

Lemma~\ref{lem-peg} shows that a peg class that can sort has at most quadratic sorting time. 
We now turn our attention to classes that can be expressed as unions of two or more peg classes. 
The goal is to show that even if the individual peg classes cannot sort, as long as their union can 
sort, it has at most quadratic sorting time. 

\begin{lemma}\label{lem-union-peg}
Assume that a class $\cC$ is a subclass of $\cRR\cup\cS_{\le k}$ for some $k$, but is not a 
subclass of $\RFringe{k}$ for any~$k$, and that a class $\cD$ is a subclass of $\RFringe{k}$ for some 
$k$, but not a subclass of~$\cRR\cup\cS_{\le k}$ for any~$k$. Then $\wst{\cC\cup\cD}{n}=O(n^2)$.
\end{lemma}
\begin{proof}
By assumption, $\cC$ contains infinitely many permutations that belong to $\cRR$ but not to 
$\RFringe{0}=\Av(21)\cup\Av(12)$. Any such permutation has the form $\iota_a\ominus\iota_b$ or 
$\delta_a\oplus\delta_b$ for some $a,b>0$. It follows that $\cC$ contains as a subclass the peg 
class 
$\Grid\left(\begin{smallmatrix} \one&\\&\Inc\end{smallmatrix}\right)$ or one of its symmetries.

Similarly, the class $\cD$ contains, for some $k$, infinitely many permutations belonging to 
$\RFringe{k}$ but not to~$\cRR$. From this, we easily deduce that $\cD$ contains as a subclass the 
peg class $\Grid\left(\begin{smallmatrix}&&\Inc\\ \one&&\\ &\one&\end{smallmatrix}\right)$ or one of 
its symmetries.

We may then easily check that any permutation from the class $\cF_n$ can be obtained by composing 
$O(n)$ permutations from $\cC\cup\cD$, implying that $\wst{\cC\cup\cD}{n}=O(n^2)$.
\end{proof}

Combining the previous arguments, we reach the main result of this section.
\begin{proposition}
A class $\cC$ that can sort satisfies $\wst{\cC}{n}=O(n^2)$.
\end{proposition}
\begin{proof}
Suppose that $\cC$ can sort. If $\cC$ contains a monotone juxtaposition or any symmetry of the 
class $\cF$ as a subclass, then its worst-case sorting time is~$O(n)$. Otherwise, $\cC$ can be 
written as a union of finitely many peg classes $\cC_1,\dotsc,\cC_t$, as shown by Albert et al. and 
Homberger--Vatter~\cite{AABRV,HomVat}. If any of the peg classes $\cC_i$ can sort, then 
$\wst{\cC}{n}\le\wst{\cC_i}{n}=O(n^2)$ by Lemma~\ref{lem-peg}. Suppose then that none of the peg 
classes $\cC_1,\dotsc,\cC_t$ can sort, and in particular, by Proposition~\ref{pro-cannot}, each 
$\cC_i$ is either a subclass of $\cRR\cup\cS_{\le k}$ or a subclass of $\RFringe{k}$ for some~$k$. 

If all of $\cC_1,\dotsc,\cC_t$ are subclasses of $\RFringe{k}$, or all of them are 
subclasses of $\cRR\cup\cS_{\le k}$, then $\cC$ cannot sort by Proposition~\ref{pro-cannot}. 
Otherwise there is a $\cC_i$ which is not a subclass of $\cRR\cup\cS_{\le k}$ for any $k$,  and a 
$\cC_j$ which is not a subclass of $\RFringe{k}$ for any $k$, and by Lemma~\ref{lem-union-peg}, 
$\wst{\cC}{n}\le\wst{\cC_i\cup\cC_j}{n}=O(n^2)$.
\end{proof}

\section{From quadratic to linear}\label{sec-lin}

Recall that $\cT$ denotes the class of cyclic adjacent transpositions. To show there is a gap 
between quadratic and linear worst-case sorting times, and to describe the 
classes with quadratic time, we need to introduce another concept: the \emph{reduced inversion 
number} of a permutation $\pi$, denoted $\rin(\pi)$, is the smallest integer $k\in\bbN_0$ such that 
$\pi$ belongs to $\cT^{\circ k}\circ\cRR$. 

The reduced inversion number $\rin(\pi)$ is inspired by the (better-known) \emph{inversion number} 
of a permutation $\pi$, which can be defined as the smallest $k$ such that $\pi$ belongs 
to~$\cCbub^{\circ k}$. Observe that the inversion number is an upper bound for the reduced inversion 
number.

%Intuitively, $\rin(\pi)\le k$ means that can be 
%transformed by a rotation and an optional reversal to a permutation with at most $k$ inversions. 
%VJ: the above needs to account for CYCLIC transpositions

\begin{lemma}\label{lem-rotrin}
 For any $\alpha,\beta\in\cRR_n$ and any $\pi\in \cS_n$, 
$\rin(\alpha\circ\pi\circ\beta)=\rin(\pi)$. 
\end{lemma}
\begin{proof}
First, notice that $\cRR\circ\cT=\cT\circ\cRR$; in other words, a permutation obtained by composing 
a rotation (with optional reversal) with a cyclic transposition can also be obtained by performing 
a cyclic transposition first followed by a rotation (with optional reversal), and vice versa. It is 
also easy to see that $\cRR\circ\cRR=\cRR$, that is, the class $\cRR$ is closed under compositions.

Suppose now that we are given $\alpha$, $\beta$ and $\pi$ as above, and let $k:=\rin(\pi)$. Let 
$\sigma$ be the permutation $\alpha\circ\pi\circ\beta$. Then $\pi$ belongs to $\cT^{\circ 
k}\circ\cRR$, and hence $\sigma$ belongs to $\cRR\circ\cT^{\circ k}\circ\cRR^{\circ 2}=\cT^{\circ 
k}\circ\cRR$. This shows that $\rin(\sigma)\le \rin(\pi)$. Conversely, $\pi$ is equal to 
$\alpha^{-1}\circ\sigma\circ\beta^{-1}$, and since $\alpha^{-1}$ and $\beta^{-1}$ both belong to 
$\cRR$, we may repeat the above argument to show that $\rin(\pi)\le \rin(\sigma)$. We conclude that 
$\rin(\pi)=\rin(\sigma)$, as claimed.
\end{proof}

\begin{lemma}\label{lem-rincompose}
For any $\sigma,\pi\in \cS_n$, 
\[
\rin(\pi\circ\sigma)\le \rin(\pi) + \rin(\sigma).
\]
\end{lemma}
\begin{proof}
 Let $k:=\rin(\sigma)$ and $\ell:=\rin(\pi)$. Then $\sigma\circ\pi$ belongs to the class 
$\cT^{\circ k}\circ\cRR\circ\cT^{\circ\ell}\circ\cRR=\cT^{\circ(k+\ell)}\circ\cRR$, hence 
$\rin(\sigma\circ\pi)\le k+\ell$.
\end{proof}

For a class $\cC$ and $n\in\bbN$, let us define $\rin(\cC;n):=\max\{\rin(\sigma);\; 
\sigma\in\cC_n\}$. As a direct consequence of Lemma~\ref{lem-rincompose}, we obtain the following 
relationship between sorting time and reduced inversion number.

\begin{corollary}\label{cor-rincompose}
For any class $\cC$ and any $\pi\in \cS_n$,
\[
\st{\cC}{\pi}\ge \frac{\rin(\pi)}{\rin(\cC;n)}.
\]
\end{corollary}

To make use of this corollary, we need to show that there are permutations $\pi\in\cS_n$ of 
large~$\rin(\pi)$. 

\begin{lemma}\label{lem-maxrin}
We have $\rin(\cS;n)=\Theta(n^2)$.
\end{lemma}
\begin{proof}
A permutation $\pi$ of length $n$ has inversion number at most $\binom{n}{2}$, and hence 
$\rin(\cS;n)\le\binom{n}{2}$. For a lower bound, we first introduce some auxiliary notions.
For two numbers $i,j\in[n]$, we define their \emph{cyclic distance} $\cd(i,j)$ by
\[
 \cd(i,j)=\min\{|i-j|, n-|i-j|\},
\]
and for a permutation $\pi\in\cS_n$, define its \emph{total cyclic distance} $\tcd(\pi)$ by
\[
 \tcd(\pi)=\cd\big(\pi(1),\pi(n)\big)+\sum_{i=1}^{n-1}\cd\big(\pi(i),\pi(i+1)\big).
\]
Observe that for any $\rho\in\cRR$ and $\pi\in\cS_n$, both $\tcd(\pi\circ\rho)$ and 
$\tcd(\rho\circ\pi)$ are equal to $\tcd(\pi)$. Observe also that for any $\sigma\in\cT$, we have 
$\tcd(\sigma\circ\pi)\le \tcd(\pi)+4$. Consequently, given that the identity permutation $\iota_n$ 
satisfies $\tcd(\iota_n)=n$, we see that for any $\pi\in\cT^{\circ k}$ we have $\tcd(\pi)\le n+4k$. 
Combining these observations, we see that any $\pi\in\cS_n$ satisfies $\tcd(\pi)\le n+4\rin(\pi)$. 

To prove the desired lower bound $\rin(\cS,n)=\Omega(n^2)$, it is enough to find a
permutation $\pi\in\cS_n$ with $\tcd(\pi)=\Omega(n^2)$. An example can be the permutation
$\pi=(1,\lceil n/2\rceil+1,2,\lceil n/2\rceil+2,3,\lceil n/2\rceil+3,\dotsc)$ obtained by 
interleaving the two sequences $(1,2,3,\dotsc,\lceil 
n/2\rceil)$ and $(\lceil n/2\rceil+1, \lceil n/2\rceil+2,\dotsc,n)$. 
\end{proof}

Combining Corollary~\ref{cor-rincompose} and Lemma~\ref{lem-maxrin}, we obtain the following bound. 

\begin{corollary}\label{cor-lowrin}
For any class $\cC$,
\[
\wst{\cC}{n}\ge\Omega\left(\frac{n^2}{\rin(\cC;n)}\right).
\]
\end{corollary}

For a gridding matrix $M$ with $m$ rows and $n$ columns, the \emph{reversal} of $M$, denoted $M^r$, 
is the gridding matrix obtained by replacing each individual entry $\cC=M_{i,j}$ with its 
reversal $\cC^r$, and then reversing the order of columns of the whole matrix $M$; formally, if 
$M_{i,j}=\cC$, then $M^r_{n-i+1,j}=\cC^r$. Note that the permutations in the class $\Grid(M^r)$ are 
precisely the reverses of those in $\Grid(M)$.

Let $M$ be a gridding matrix. The \emph{column rotation} of $M$ is the operation that transforms 
$M$ into the matrix $M'$ of the same size, where the first column of $M$ is equal to the last column 
of $M'$, and for each $i>1$ the $i$-th column of $M$ is equal to the $(i-1)$-th column of~$M'$. The 
\emph{row rotation} of $M$ is defined analogously. We say that a gridding matrix $M$ is 
\emph{congruent} to a matrix $M'$, if $M$ can be transformed into $M'$ by a sequence of column 
rotations, row rotations and reversals. Note that if $M$ and $M'$ are congruent, then for any 
$\pi\in\Grid(M)$ there are permutations $\rho,\rho'\in\cRR$ such that $\rho\circ\pi\circ\rho'$ 
belongs to $\Grid(M')$. In particular, $\rin(\Grid(M);n)=\rin(\Grid(M');n)$.

\begin{lemma}\label{lem-pegrin}
For any peg class $\cC$, either $\rin(\cC;n)$ is bounded, or $\wst{\cC}{n}=O(n)$. 
\end{lemma}
\begin{proof}
Let $\cC$ be a peg class with a gridding matrix~$M$, and suppose that $\rin(\cC;n)$ is unbounded. 
In particular, $M$ contains at least one entry equal to $\Inc$ or to~$\Dec$. Let $M'$ be a gridding 
matrix congruent to $M$ whose top-right entry is equal to~$\Inc$. 

Suppose that $M'$ contains an entry equal to $\Dec$. Then $\Grid(M')$ contains, as a subclass, the 
peg class $\cCpan=\Grid \left(\begin{smallmatrix} &\Inc\\\Dec&\end{smallmatrix}\right)$, and $\cC$ 
contains one of the symmetries of this class. Since all the symmetries of $\cCpan$ have linear 
worst-case sorting time, we conclude that $\wst{\cC}{n}=O(n)$.

\begin{figure}
\begin{align*}
\cCins&=\Grid\begin{pmatrix}
        &&\Inc\\\one&&\\&\Inc&
       \end{pmatrix}, &
\cC_1&=\Grid\begin{pmatrix}
        \one&&\\&\Inc&\\&&\Inc
       \end{pmatrix}, &
\cC_2&=\Grid\begin{pmatrix}
        \Inc&&\\&\one&\\&&\Inc
       \end{pmatrix}
\end{align*}
\caption{Three classes with congruent gridding matrices, all with linear sorting 
time}\label{fig-cong}
\end{figure}

Suppose then that $M'$ does not contain~$\Dec$. Let $M'_{i,j}$ be an entry of $M'$ equal 
to~$\Inc$. If $M'$ contains a nonempty entry strictly to the left and above $M'_{i,j}$, then 
$\Grid(M')$ contains $\cCins$ as a subclass. It follows that $\cC$ contains as a subclass one of 
the three classes from Figure~\ref{fig-cong} or one its symmetries. Noting that $\cCins\subseteq 
\cC_1^{\circ 2}$ and $\cCins\subseteq\cC_2^{\circ 2}$, and applying Observation~\ref{obs-r} if 
needed, we easily conclude that $\wst{\cC}{n}=O(n)$.

An analogous argument applies when $M'$ has a nonzero entry to the right and below $M'_{i,j}$.

Finally, suppose that no entry of $M'$ equal to $\Inc$ has any nonzero entry to the left and 
above, or to the right and below of itself. It follows that all the entries of $M'$ equal to $\Inc$ 
appear on the main diagonal, and $M'$ is a block-diagonal matrix whose every block consists either 
of a single entry $\Inc$, or a submatrix that has no entry equal to~$\Inc$. It then easily follows 
that $\rin(\Grid(M');n)$ is bounded, and $\rin(\cC;n)$ is bounded as well. 
\end{proof}

\begin{proposition}\label{pro-kvadr}
For a class $\cC$, the following are equivalent:
\begin{enumerate}
\item $\wst{\cC}{n}=o(n^2)$, 
\item $\wst{\cC}{n}=O(n)$, and
\item $\rin(\cC;n)$ is unbounded as $n\to\infty$.
\end{enumerate}
\end{proposition}
\begin{proof}
By Corollary~\ref{cor-lowrin}, if $\rin(\cC;n)$ is bounded then $\wst{\cC}{n}=\Omega(n^2)$. This 
shows the implication $1\implies 3$. 

To prove $3\implies 2$, suppose now that $\rin(\cC;n)$ is unbounded for a class~$\cC$.  If $\cC$ 
contains as a subclass a monotone juxtaposition or its inverse or if it contains the Fibonacci class 
$\cF$ or its reverse, then clearly $\wst{\cC}{n}=O(n)$. On the other hand, if $\cC$ contains no 
such 
subclass, then it can be expressed as a union of finitely many peg classes $\cC_1,\dotsc,\cC_k$. 
Since $\rin(\cC;n)$ is unbounded, there must be an $i\in[k]$ such that $\rin(\cC_i;n)$ is unbounded. 
By Lemma~\ref{lem-pegrin}, $\wst{\cC_i}{n}=O(n)$, and therefore $\wst{\cC}{n}=O(n)$. This proves the 
implication $3\implies 2$.  

The implication $2\implies 1$ is trivial.
\end{proof}

\section{From \texorpdfstring{$\Omega(\sqrt n)$}{Ω(√n)} to polylog}\label{sec-polylog}

\newcommand{\bbX}{\mathfrak{X}}

The most difficult part of our argument deals with the gap between sorting times
$O(\log^2 n)$ and $\Omega(\sqrt n)$. The criterion distinguishing the two types of classes is rather 
simple to state, though. For ease of notation, let $\bbX$ be the set of permutation classes that 
contains the eight monotone juxtapositions and the four classes $\cL$, $\cL^r$, $\PBT$ and $\PBT^r$.
We say that a permutation class $\cC$ is \emph{$\bbX$-avoiding} if it does not contain any class 
from $\bbX$ as a subclass, otherwise we say that it is $\bbX$-containing.

\begin{theorem}\label{thm-lin}
Any $\bbX$-avoiding class $\cC$ satisfies $\wst{\cC}{n}=\Omega(\sqrt n)$, while any 
$\bbX$-containing 
class $\cC$ satisfies $\wst{\cC}{n}=O(\log^2 n)$.
\end{theorem}

The rest of this section is devoted to a detailed proof of Theorem~\ref{thm-lin}. 
We first focus on showing that every $\bbX$-avoiding class $\cC$ satisfies 
$\wst{\cC}{n}=\Omega(\sqrt n)$. 

An \emph{inversion} in a permutation $\pi$ is a decreasing subsequence of length 2, i.e., a pair 
$(\pi(i),\pi(j))$ such that $i<j$ and $\pi(i)>\pi(j)$. The first ingredient in the proof of Theorem~\ref{thm-lin} is a structural result characterizing permutation classes admitting a linear bound on the number of inversions.

\begin{proposition}\label{pro-inversion}
For a permutation class $\cC$, the following statements are equivalent:
\begin{enumerate}
 \item Neither $\cR$ nor $\Dec$ is a subclass of $\cC$.
 \item There is a $k\in\bbN$ such that every permutation in $\cC$ avoids both $\delta_k$ and $\iota_k\ominus\iota_k$.
 \item There is a $Q\in\bbN$ such that every permutation $\pi\in\cC$ of size $n$ has at most $Qn$ inversions.
\end{enumerate}
\end{proposition}
\begin{proof}
 The implication $3\Rightarrow 1$ is clear, since both $\Dec$ and $\cR$ contain permutations of size $n$ with $\Omega(n^2)$ inversions, namely $\delta_n$ and $\iota_{\lfloor n/2\rfloor}\ominus\iota_{\lceil n/2\rceil}$. Likewise, the implication $1\Rightarrow 2$ follows directly from the fact that any permutation in $\cR$ is a subpermutation of $\iota_k\ominus\iota_k$ for some~$k$.
 
It remains to prove $2\Rightarrow 3$. Fix $k$ such that every permutation of $\cC$ avoids both $\delta_k$ and $\iota_k\ominus\iota_k$, and let $\pi$ be a permutation of size $n$ in~$\cC$. Since $\pi$ avoids $\delta_k$, it can be partitioned into $k-1$ increasing subsequences $I_1,\dotsc,
I_{k-1}$. We will show that $\pi$ has at most $(k-1)^3n$ inversions.

Suppose that $(\pi(i),\pi(j))$ is an inversion, that is, $i<j$ and 
$\pi(i)>\pi(j)$. We say that the inversion has \emph{type $(a,b)$} if $\pi(i)$ belongs to $I_a$ and 
$\pi(j)$ belongs to $I_b$. Note that this implies $a \neq b$. We will show that for any 
$(a,b)\in[k-1]\times[k-1]$ there are at most $(k-1)n$ inversions of type $(a,b)$. To this end, we 
define a graph $G_{a,b}$ on the vertex set $I_a\cup I_b$ whose edges are the pairs of vertices 
that form an inversion of type $(a,b)$. We claim that the graph $G_{a,b}$ is 
$(k-1)$-degenerate, that is, each of its nonempty subgraphs has a vertex of degree at most~$k-1$. 

Suppose for contradiction that $G_{a,b}$ has a nonempty subgraph $G'$ whose every vertex has degree 
at least~$k$. The vertex set of $G'$ has the form $I'_a\cup I'_b$ for some $I'_a\subseteq I_a$ 
and $I'_b\subseteq I_b$. Let $x_1$ be the rightmost vertex in $I'_a$, and let $N(x_1)$ be the 
vertices of $I'_b$ adjacent to~$x_1$. By assumption, $|N(x_1)|\ge k$. Let $y_1$ be the 
rightmost (and therefore also topmost) vertex in $N(x_1)$, and let $N(y_1)\subseteq I'_a$ be its 
neighbors. Since any vertex in $N(x_1)$ is to the right of $x_1$, it follows that any vertex in 
$N(x_1)$ is to the right of all the vertices in $N(y_1)$ (recall that $x_1$ is the rightmost vertex 
of $I'_a$). Moreover, since $y_1$ is topmost in $N(x_1)$, all the vertices of $N(x_1)$ are below 
all the vertices in~$N(y_1)$. It follows that the set $N(x_1)\cup N(y_1)$ induces in $\pi$ a 
permutation containing $\iota_k\ominus \iota_k$, which is a contradiction.

We conclude that the graph $G_{a,b}$ is $(k-1)$-degenerate for any $(a,b)\in[k-1]\times[k-1]$, and 
therefore has at most $(k-1)(|I_a|+|I_b|)\le (k-1)n$ edges. Summing over all values of $a$ and 
$b$, we conclude that $\pi$ has at most $(k-1)^3n$ inversions.
\end{proof}

The key part in the proof of Theorem~\ref{thm-lin} is a structural characterization of $\bbX$-avoiding classes. To state it, we need to introduce some terminology.

An \emph{integer interval} (or just \emph{interval}, if there is no risk of confusion) is a finite 
set of consecutive positive integers, i.e., a set of the form $\{i\in\bbN; a\le i\le b\}$ for some 
$a,b\in\bbN$. For a finite set $X\subseteq\bbN$, its \emph{intervalicity}, denoted $\Int(X)$, is 
the smallest $k\in\bbN_0$ such that $X$ can be expressed as a union of $k$ integer intervals. For a 
finite point set $P\subseteq\bbN\times\bbN$, we let $P|_x$ and $P|_y$ denote its projections to the 
$x$-axis and to the $y$-axis, respectively. The \emph{intervalicity} of the point set $P$, denoted 
again $\Int(P)$, is the maximum of $\Int(P|_x)$ and $\Int(P|_y)$.

\begin{proposition}\label{pro-avoid}
For every $\bbX$-avoiding class $\cC$ there is a constant $K\equiv K(\cC)$, such that every 
$\pi\in\cC$ admits a gridding with at most $K$ rows and $K$ columns such that for every cell $X$ of 
the gridding, the following holds:
\begin{enumerate}
 \item the point set $\pi[X]$ has intervalicity at most $K$,
 \item the permutation induced by $\pi[X]$ avoids at least one of $\iota_K$ and $\delta_K$,
 \item the permutation induced by $\pi[X]$ avoids both $\iota_K\ominus\iota_K$ and 
$\delta_K\oplus\delta_K$, 
and
 \item if $n_X$ is the number of elements in $\pi[X]$, then either $\pi[X]$ or its reverse has at most $K\cdot n_X$ inversions.
\end{enumerate}
\end{proposition}

Before we prove the proposition, let us collect several known facts about permutation classes.

\begin{fact}[\protect{Vatter~\cite[Corollary 5]{VatterEH}}]\label{fac-monot}
For every permutation class $\cC$ that does not contain $\cL$ or $\cL^r$ as a subclass, there is a 
constant $K=K_{\ref{fac-monot}}$ such that every permutation $\pi\in\cC$ is a union of at most 
$K$ monotone subsequences.
\end{fact}

\begin{fact}[\protect{Vatter~\cite[Theorem 3.1]{VatterSmall}}]\label{fac-sum}
Let $\cC$ be a permutation class, and let $\alpha$ be a permutation. The following two statements 
are equivalent:
\begin{itemize}
 \item There is a constant $K=K_{\ref{fac-sum}}$ such that every permutation in $\cC$ has a 
gridding of size $K\times K$ in which every cell induces a permutation that avoids~$\alpha$. 
\item There is a constant $Q$ such that neither $\bigoplus^Q\alpha$ nor $\bigominus^Q\alpha$ 
belongs to~$\cC$.
\end{itemize}
\end{fact}

A \emph{horizontal alternation} is a permutation in which all the odd values appear to the left of 
all the even values or vice versa. A \emph{vertical alternation} is an inverse of a horizontal 
alternation. An \emph{alternation} is a horizontal or vertical alternation.

\begin{fact}[\protect{Huczynska--Vatter~\cite[Proposition 3.2]{HV}}]\label{fac-alt}
For a permutation class $\cC$ that does not contain any monotone juxtaposition as a subclass there 
is a constant $K=K_{\ref{fac-alt}}$ such that $\cC$ does not contain any alternation of size~$K$.
\end{fact}

\begin{proof}[Proof of Proposition~\ref{pro-avoid}]
Let $\cC$ be an $\bbX$-avoiding class, and fix $\pi\in \cC$. 

By Fact~\ref{fac-alt}, there is a constant $K_{\ref{fac-alt}}$ such that $\cC$ does not contain 
any alternation of size $K_{\ref{fac-alt}}$ or more. We claim that this implies that in any gridding 
of $\pi$, every cell has intervalicity less than~$K_{\ref{fac-alt}}$. To see this, suppose for 
contradiction that there is a cell $X$ such that the projection of $\pi[X]$ on the horizontal or 
vertical axis has intervalicity at least~$K_{\ref{fac-alt}}$. Without loss of generality, suppose 
that the projection is onto the horizontal axis, let $\ell\ge K_{\ref{fac-alt}}$ be its 
intervalicity, and let $P_1, P_2,\dotsc, P_\ell$ be the intervals, numbered left to right, whose 
union is the projection of~$\pi[X]$. For any $i\in[\ell-1]$, fix an integer $j_i$ such that $\max 
P_i<j_i<\min P_{i+1}$. Since $\pi(j_i)$ is not in $\pi[X]$, it is either smaller than all the 
values in $\pi[X]$ or larger than all the values in~$\pi[X]$. Without loss of generality, there 
are at least $(\ell-1)/2$ values of $i\in[\ell-1]$ for which $\pi(j_i)$ is larger than all the 
values in $\pi[X]$. We can then choose at least $(\ell+1)/2$ values in $\pi[X]$ which form a 
vertical alternation with these $(\ell-1)/2$ `large' values $\pi(j_i)$, showing that $\pi$ has an 
alternation of size $\ell\ge K_{\ref{fac-alt}}$, which is impossible. Thus, in any gridding of 
$\pi$, every cell has intervalicity less than~$K_{\ref{fac-alt}}$. 

By Fact~\ref{fac-monot}, there is a constant $K_{\ref{fac-monot}}$ such that every permutation of 
$\cC$ is a union of at most $K_{\ref{fac-monot}}$ monotone sequences. Let us therefore fix such a 
partition of $\pi$ into monotone sequences, and suppose that the partition contains 
$p$ increasing sequences $I_1,\dotsc, I_p$ and $q$ decreasing sequences $D_1,\dotsc,D_q$, for some 
$p+q\le K_{\ref{fac-monot}}$. Define $I_\cup=\bigcup _{i=1}^p I_i$ and $D_\cup=\bigcup_{i=1}^q D_i$. 
Notice that $I_\cup$ avoids $\delta_{p+1}$ and $D_\cup$ avoids~$\iota_{q+1}$. 

We will now show that $\pi$ has a gridding with bounded number of cells in which each cell is 
either disjoint from $I_\cup$ or disjoint from $D_\cup$. For any choice of 
$i\in[p]$ and $j\in[q]$, we can easily construct a $2\times 2$ gridding of $\pi$ such that the 
increasing subsequence $I_i$ intersects neither the top-left nor the bottom-right cell, while $D_j$ 
intersects neither the bottom-left nor the top-right cells. Combining such griddings for all the 
choices of $(i,j)\in[p]\times[q]$, we obtain a single gridding of size $(pq+1)\times (pq+1)$ in 
which every cell avoids either $I_\cup$ or $D_\cup$, and in particular, every cell avoids 
$\iota_{q+1}$ or~$\delta_{p+1}$. 

For the next step, we note that there are constants $A,B$ such that neither $\bigoplus^A 
(\iota_B\ominus\iota_B)$ nor $\bigominus^A (\iota_B\ominus\iota_B)$ belongs to~$\cC$. Indeed, if 
$\cC$ contained $\bigoplus^A (\iota_B\ominus\iota_B)$ for each $A$ and $B$, then $\cC$ would 
contain $\PBT$ as a subclass, and if $\cC$ contained $\bigominus^A(\iota_B\ominus\iota_B)$ for all 
$A,B$, then $\cC$ would contain $\cL^r$ as a subclass. By Fact~\ref{fac-sum} applied to 
$\alpha=\iota_B\ominus\iota_B$, we can refine the gridding of $\pi$ obtained in the previous step 
by adding at most $K_{\ref{fac-sum}}$ rows and columns to obtain a gridding in which no cell  
contains $\iota_B\ominus\iota_B$. By a symmetric argument, another round of refinement ensures that 
no cell of the gridding contains $\delta_{\overline B}\oplus\delta_{\overline B}$ for a
constant~$\overline B$.

Invoking Proposition~\ref{pro-inversion} (or its symmetric version) to each cell of the resulting gridding, and choosing $K$ large enough in terms of the constants appearing in the previous arguments, we conclude that the gridding obtained by this procedure has all the properties from the statement of the proposition.
\end{proof}

For a permutation $\pi$, the \emph{adjacency graph} $G(\pi)$ is a graph whose vertices are the 
elements of $\pi$, and two vertices are connected by an edge if and only if the corresponding two 
elements have adjacent positions or adjacent values. In particular, the adjacency graph is a union 
of two paths, one visiting the vertices in left-to-right order, and the other in bottom-to-top 
order.

The \emph{tree-width} of the permutation $\pi$, denoted $\tw(\pi)$, is then defined as the 
tree-width of $G(\pi)$. Note that there are permutations of size $n$ with tree-width 
$\Omega(n)$~\cite{Ahal2008}.

Now, we introduce the concept of \emph{sorting diagram}, which will offer a convenient way to 
represent a sequence of steps performed during the sorting of a sequence. Let $t$ and $n$ 
be positive integers, and let $\vec\sigma=(\sigma^1, \sigma^2,\dotsc,\sigma^t)$ be a $t$-tuple of 
permutations of $[n]$. 
Define the permutation $\pi=\sigma^t\circ\sigma^{t-1}\circ\dotsb\circ\sigma^2\circ\sigma^1$. In 
particular, by applying the steps $\sigma^1,\dotsc,\sigma^t$ in this order, we would sort 
the sequence $\pi^{-1}$ into an increasing sequence. 

Let $P$ be the set $[(t+1)n]\times [(t+1)n]$, viewed as an integer grid in the plane.  We partition 
$P$ into blocks of size $n\times n$ as follows: for $a,b\in[t+1]$, let $P_{a,b}$ be the set 
$\{(i,j)\in P; (a-1)n<i\le an \land  (b-1)n<j\le bn \}$. 
For a permutation $\sigma\in\cS_n$, let $P_{a,b}[\sigma]$ denote the set of $n$ points forming a 
translated copy of the diagram of $\sigma$ inside the block $P_{a,b}$; formally,
\[
P_{a,b}[\sigma] =\left\{ \big( (a-1)n+i, (b-1)n+\sigma(i) \big); i\in[n]\right\}.
\]
\begin{figure}
\centerline{\includegraphics[width=0.5\textwidth]{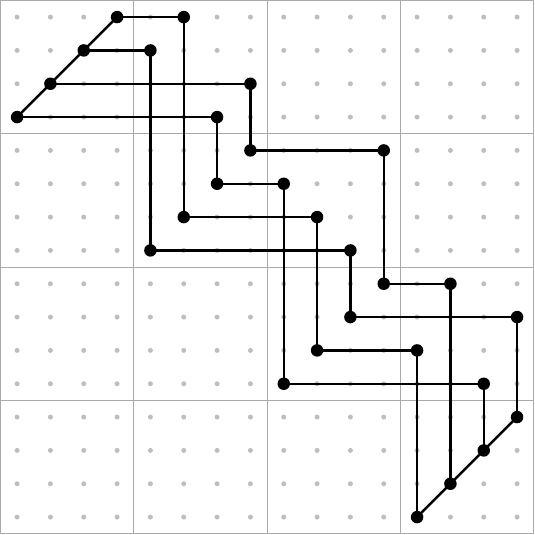}}
\caption{The sorting diagram of the triple of permutations $\sigma^1=2413, \sigma^2=3214, 
\sigma^3=3412$.}\label{fig-sortdia}
\end{figure}
\newcommand{\SD}{SD}
We are now ready to describe the sorting diagram of the $t$-tuple of permutations  
$\vec\sigma=(\sigma^1, \sigma^2,\dotsc,\sigma^t)$; refer to Figure~\ref{fig-sortdia}. Recall that 
$\iota_n$ is the increasing permutation $(1,2,\dotsc,n)$.
The \emph{sorting diagram} of the $t$-tuple $\vec\sigma=(\sigma^1,\dotsc,\sigma^t)$, denoted 
$\SD\equiv\SD(\vec\sigma)$, is the graph whose vertex set is the point set 
\[
V(\SD)=\bigcup_{i=1}^{t+1} P_{t+2-i,i}[\iota_n] \cup \bigcup_{i=1}^t P_{t+2-i,i+1}[\sigma^i].
\]
In particular, the vertex set consists of $t+1$ copies of the diagram of $\iota_n$ inside the blocks 
forming the decreasing diagonal of $P$, together with the copies of the diagrams of 
$\sigma^1,\dotsc,\sigma^t$, appearing right-to-left inside the blocks just above the decreasing 
diagonal. We refer to the $n$ vertices inside the bottom-right block $P_{t+1,1}$ as \emph{starting 
vertices}, and we label them $s_1,\dotsc,s_n$ in left-to-right order. Similarly, the vertices inside 
the top-left block $P_{1,t+1}$ are the \emph{terminal vertices}, and are labelled $t_1,\dotsc,t_n$ 
in left-to-right order.
The edge set of $\SD$ is determined as follows:
\begin{itemize}
\item Any two vertices of $V(\SD)$ that lie on the same horizontal line are connected by an edge. We 
call these edges \emph{horizontal edges}. Note that every vertex is incident to a unique horizontal 
edge, except for the $n$ starting vertices.
\item Any two vertices of $V(\SD)$ that lie on the same vertical line are connected by an edge. We 
call these edges \emph{vertical edges}. Note that every vertex is incident to a unique vertical
edge, except for the $n$ terminal vertices. Note also that the horizontal and vertical edges 
together form $n$ vertex-disjoint paths, each connecting a starting vertex to a distinct terminal 
vertex, with the vertex $s_i$ being connected to $t_{\pi(i)}$, where 
$\pi=\sigma^t\circ\dotsb\circ\sigma^1$. 
\item Finally, for any $i\in[n-1]$, the vertex $s_i$ is connected to $s_{i+1}$ and $t_i$ to 
$t_{i+1}$. We call these edges as the \emph{diagonal edges}.
\end{itemize}
This completes the description of the graph $\SD$.

\begin{lemma}\label{lem-twsd}
For a $t$-tuple of permutations $\vec\sigma=(\sigma^1,\dotsc,\sigma^t)$ and their composition 
$\pi=\sigma^t\circ\dotsb\circ\sigma^1$, the tree-width of the graph $\SD(\vec\sigma)$ is at least as 
large as the tree-width of the incidence graph $G(\pi)$ of~$\pi$.
\end{lemma}
\begin{proof}
The graph $G(\pi)$ is a minor of $\SD(\vec\sigma)$, since by contracting all the horizontal and 
vertical edges in $\SD(\vec\sigma)$, we obtain precisely the graph $G(\pi)$. The lemma follows, 
since tree-width is minor-monotone.
\end{proof}

\begin{proposition}\label{pro-embed}
Let $\cC$ be an $\bbX$-avoiding class. There are constants $Q$ and $R$, such that for any 
$t\in\bbN$, any $n\in\bbN$, and any $t$-tuple $\vec\sigma=(\sigma^1,\dotsc,\sigma^t)$ of 
permutations from $\cC_n$, the sorting diagram $\SD(\vec\sigma)$ can be drawn on a surface of genus 
at most $Q\cdot t$ with at most $R\cdot t\cdot n$ edge-crossings.
\end{proposition}
\begin{proof}
The definition of the graph $\SD=\SD(\vec\sigma)$ already implies a drawing in the plane, where 
vertices are points and edges are straight-line segments. Unfortunately, such a drawing may have too 
many edge-crossings. Our strategy is to insert handles and cross-caps into this drawing, so that we 
eliminate enough of the crossings, without increasing the genus of the surface too much. 

Recall from the definition of sorting diagram that the vertices of $\SD$ are organized into $2t+1$ 
square blocks, which include $t+1$ diagonal blocks $P_{t+1,1},P_{t,2},\dotsc,P_{1,t+1}$, as well as 
$t$ blocks right above the diagonal, namely $P_{t+1,2},P_{t,3},\dotsc,P_{2,t+1}$. For ease of 
notation, for $i\in[t+1]$ let $D_i$ refer to the diagonal block $P_{t+2-i,i}$, and for $i\in[t]$ 
let $S_i$ refer to the block right above $D_i$, i.e., $P_{t+2-i,i+1}$. The vertices of $\SD$ induce 
a copy of $\iota_n$ inside every $D_i$, and a copy of $\sigma^i$ inside~$S_i$. 

In the straight-line drawing of $\SD$, the only edge-crossings that occur are inside the 
blocks $S_i$, between a vertical edge connecting a vertex $x$ from $D_i$ to a vertex $y$ in $S_i$, 
and a horizontal edge connecting a vertex $x'$ in $S_i$ to a vertex $y'$ in $D_{i+1}$. Notice that 
such a pair of edges has a crossing if and only if $(y,x')$ corresponds to a decreasing subsequence 
of~$\sigma^i$. We will now modify the drawing of edges (and the surface they are drawn upon) inside 
each block $S_i$ for $i\in[t]$, while the parts of the drawing inside the $D_i$'s are unaffected. 

\begin{figure}
 \centerline{\includegraphics{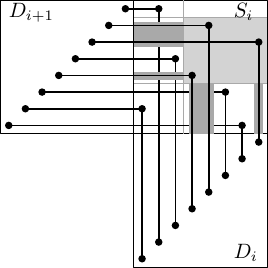}}
 \caption{A fragment of a sorting diagram, induced by three consecutive blocks $D_i$, $S_i$ and 
$D_{i+1}$. The light gray rectangle is a cell of the gridding of $\sigma^i$. The four dark gray 
rectangles represent the two horizontal and two vertical handles attached 
to this cell.}\label{fig-handles}
\end{figure}

Fix $i\in[t]$. Let $K=K(\cC)$ be the constant from Proposition~\ref{pro-avoid}. As the first step, 
apply to the copy of $\sigma^i$ inside $S_i$ the gridding described by Proposition~\ref{pro-avoid}. 
We will now add handles to our surface, and route the edges along these handles, to ensure that no 
two edges cross, unless they have endpoint inside the same cell of the gridding. Refer to 
Figure~\ref{fig-handles}.

Let $X$ be a cell of 
the gridding of $\sigma^i$. By Proposition~\ref{pro-avoid}, $\sigma^i[X]$ has intervalicity at 
most~$K$. In particular, there are $K$ disjoint intervals $I_1,\dotsc,I_K$ (some possibly empty) 
such that a vertical edge from $D_i$ to $S_i$ has an endpoint in $X$ if and only if its projection 
to the $x$-axis is a point from $I_1\cup I_2\cup\dotsb\cup I_K$. For each $j\in[K]$ such that 
$I_j\neq\emptyset$, we attach a new handle $H^|_j$ to our surface, oriented vertically, whose bottom 
end is near the bottom edge of $S_i$, the top end is near the bottom edge of $X$, and its horizontal 
position corresponds to the interval $I_j$ (in case $X$ is in the bottom row of the gridding, we omit 
these handles and leave the vertical edges into $X$ unchanged). All the vertical edges whose 
projection on the $x$-axis lies in $I_j$ are then routed through $H^|_j$, from the point when they 
cross the boundary from $D_i$ to $S_i$, till the point when they reach the bottom of edge of the 
cell~$X$. This way, we remove any crossing between a vertical edge incident with a point in~$X$ and a 
horizontal edge incident with a cell $X'$ in the gridding which lies below~$X$.

We perform an analogous operation with the horizontal edges. Again, for a cell $X$, the projection 
of $\sigma^i[X]$ on the $y$-axis forms $K$ intervals $J_1,\dotsc,J_K$, and for each such nonempty
interval $J_\ell$ we attach a horizontal handle $H^-_\ell$, whose left endpoint is near the 
boundary between $S_i$ and $D_{i+1}$ and its right endpoint is near the left edge of~$X$; if $X$ is 
in the leftmost column of the gridding, these handles are omitted. The 
horizontal edges projecting into $J_\ell$ are then routed through the handle $H^-_\ell$. 

For each $i\in[t]$, this operation adds at most $K$ vertical and $K$ horizontal handles 
incident to any given cell $X$ of the gridding of $\sigma^i$, so at most $2K^3t$ handles overall.
After these handles are added, a vertical edge from $D_i$ to $S_i$ may cross a horizontal edge 
from $S_i$ to $D_{i+1}$ only if their endpoints inside $S_i$ belong to the same cell of the 
gridding of~$\sigma^i$. Unfortunately, there may still be too many such edge-crossings, and so we 
may need to perform additional modifications inside a cell $X$ to reduce the number of crossings. 
The goal is for the number of crossings inside $X$ to be linear in the size of~$\sigma^i[X]$.

\begin{figure}
 \centerline{\includegraphics{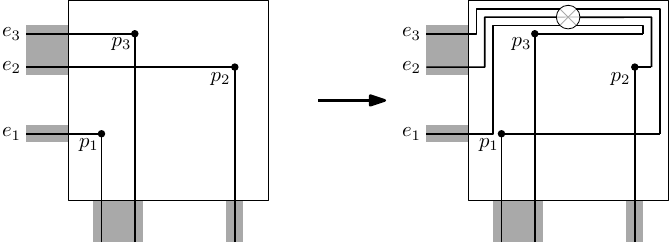}}
 \caption{A modification of horizontal edges inside a cell $X$ of the 
gridding of $\sigma^i$. The gray rectangles represent the handles adjacent to $X$, 
and the circle near the top boundary of the modified drawing represents a 
cross-cap.}\label{fig-uturn}
\end{figure}

Consider again a cell $X$ of the gridding of $\sigma^i$, and let $n_X$ be the size of 
$\sigma^i[X]$. Observe that at this point, the number of edge-crossings inside $X$ is equal to the 
number of decreasing subsequences of length 2 in $\sigma^i[X]$. By Proposition~\ref{pro-avoid}, 
$\sigma^i[X]$ contains either at most $K\cdot n_X$ decreasing subsequences of length 2, or at most 
$K\cdot n_X$ increasing subsequences of length 2. In the former case, we do not need to modify the 
drawing of the edges incident to $X$ any further. Consider now that the second situation occurs; 
refer to Figure~\ref{fig-uturn}.

Let $p_1, \dotsc,p_{n_X}$ be the points of $\sigma^i[X]$, numbered bottom to top, and let $e_j$ be 
the horizontal edge incident to~$p_j$.  
We will modify the drawings of all the horizontal edges inside $X$ as follows: the edge $e_j$, 
instead of heading straight to the left from $p_j$ towards the horizontal handle leading to 
$D_{i+1}$, will instead head straight to the right from $p_j$, until it reaches close to the right 
edge of $X$, then it will turn upwards, follow the right edge towards the top-right corner of $X$, 
then turn left and follow the top edge of $X$. All the horizontal edges will perform this `U-turn', 
keeping their mutual positions so that they do not introduce any mutual crossings. 
In particular, as the edges reach the top edge of $X$ and head to the left, the edge $e_1$ will be 
the topmost, followed by $e_2$, $e_3$ etc. However, we need the edges to reach the handles in the 
opposite bottom-to-top order. To fix this, we attach a cross-cap near the top edge of $X$ and let 
all the edges $e_i$ pass through it. This reverses their bottom-to-top order, and it is easy to then 
extend the edges along the top and left edge of $X$ towards their designated handles, without 
introducing any new crossings. After this modification, an edge $e_j$ crosses a vertical edge 
incident to a vertex $p_k$ if and only if $(p_j,p_k)$ form an increasing subsequence of length 2 in 
$\sigma^i[X]$. In particular, there are now at most $K\cdot n_X$ crossings inside~$X$.

In the end, we obtain a drawing of $\SD(\vec\sigma)$ on a surface with at most $2K^3t$ handles and 
at most $K^2t$ cross-caps, and in particular, the genus of the surface is at most $(4K^3+K^2)t$. 
Each block $S_i$ contains at most $Kn$ edge-crossings, so there are at most $Knt$ edge-crossings 
overall.
\end{proof}

Our next tool is an inequality that relates the genus, the crossing number and the tree-width of a 
graph. Its proof is based an idea of Dujmovi\'c, Eppstein and Wood~\cite{DEW}.

\begin{proposition}\label{pro-gctw}
Every graph on $n$ vertices that can be drawn on a surface of genus $g$ with at most $x$ 
edge-crossings has treewidth $O\left(\sqrt{(g+1)(n+x)}\right)$.
\end{proposition}
\begin{proof}
Suppose a graph $G$ on $n$ vertices has been drawn on a surface of genus $g$ with $x$ 
edge-crossings. Let us orient the edges of $G$ arbitrarily, so that each edge $e$ has a designated 
tail $t_e$ and head $h_e$. Next, let us replace each edge-crossing by a new vertex, thereby 
obtaining a crossing-free drawing of a new graph $H$ with $n_H=n+x$ vertices. Since $H$ can be drawn 
on a surface of genus $g$, it has treewidth at most 
$O(\sqrt{g\cdot n_H})=O(\sqrt{g(n+x)})$~\cite{DvorakNorin,GHT}. Let us fix an optimal tree 
decomposition $(T_H,\beta_H)$ of the graph $H$, where $T_H$ is the corresponding decomposition tree, 
and $\beta_H\colon V(T_H)\to 2^{V(H)}$ is the function that assigns to each node $q$ of $T_H$ a bag 
$\beta(q)$ of vertices of $H$ of size at most $O(\sqrt{g(n+x)})$. 

We now modify this decomposition to a tree decomposition of $G$, as follows. Let $w$ be any vertex 
of $H$ that is not a vertex of $G$. This means that in the original drawing of $G$, $w$ corresponded 
to a crossing of a pair of edges $e$ and $f$. Let $t_e$ and $t_f$ be the tails of the two edges. We 
modify the decomposition $(T_H, \beta)$ by replacing any occurrence of the vertex $w$ in  any of 
the bags by the two vertices $t_e$ and~$t_f$. After we perform these replacements for every vertex 
$w\in V(H)\setminus V(G)$, we may easily check that we obtain a tree decomposition of $G$ of width 
at most twice as large as the width of $(T_H,\beta)$ (see Dujmovi\'c et al.~\cite{DEW} for details).
%TODO find theorem number
\end{proof}

We now have all the ingredients to prove the harder part  of Theorem~\ref{thm-lin}.

\begin{proposition}\label{pro-lin1}
Any $\bbX$-avoiding class $\cC$ has $\wst{\cC}{n}=\Omega(\sqrt{n})$.
\end{proposition}
\begin{proof}
Let $\pi\in\cS_n$ be a permutation of tree-width $\Omega(n)$ (such permutations exist by a result of 
Ahal and Rabinovich~\cite[Theorem~3.4 and Proposition~3.6]{Ahal2008}). Let us put 
$t:=\st{\cC}{\pi}$, and let $\vec\sigma=(\sigma^1,\dotsc,\sigma^t)$ be a $t$-tuple of permutations 
from $\cC$ such that $\pi=\sigma^t\circ\sigma^{t-1}\circ\dotsb\circ\sigma^1$. 

Consider the sorting diagram $\SD=\SD(\vec\sigma)$. By Lemma~\ref{lem-twsd}, 
$\tw(\SD)\ge\tw(\pi)=\Omega(n)$. On the other hand, $\SD$ has $(2t+1)n$ vertices, and by
Proposition~\ref{pro-embed}, it can be drawn on a surface of genus $O(t)$ with
$O(tn)$ edge-crossings. Therefore, by Proposition~\ref{pro-gctw}, $\SD$ has tree-width 
$O(t\sqrt{n})$. Combining these bounds, we conclude that
\[
\wst{\cC}{n}\ge \st{\cC}{\pi}=t=\Omega(\sqrt{n}),
\]
proving the proposition.
\end{proof}

To complete the proof of Theorem~\ref{thm-lin}, we now turn to upper bounds.

\begin{proposition}\label{pro-lin2}
Any $\bbX$-containing class $\cC$ has $\wst{\cC}{n}=O(\log ^2 n)$.
\end{proposition}
\begin{proof}
It is clearly enough to prove the proposition for a class $\cC$ that belongs to~$\bbX$. Let us 
first consider the monotone juxtapositions, starting with $\cC_1=\Grid\left(\begin{smallmatrix}\Inc\\ 
\Inc\end{smallmatrix}\right)$. We may easily observe that $\wst{\cC_1}{n}=O(\log n)$, e.g., by 
considering the elements of $[n]$ as binary strings of length $O(\log n)$ and sort them using 
radix-sort, which can be performed by $O(\log n)$  operations of $\cC_1$.

With the help of Observation~\ref{obs-r}, we conclude that also the classes $\cC_1^r$, $\cC_1^{-1}$, 
and $(\cC_1^r)^{-1}$ have worst-case sorting time $O(\log n)$.

Consider now the class $\cC_2=\Grid \left(\begin{smallmatrix}\Inc\\ \Dec\end{smallmatrix}\right)$. Notice 
that $\cC_1\subseteq \cC_2^{\circ 2}$, which implies $\wst{\cC_2}{n}=O(\log n)$ via 
Observation~\ref{obs-comp}. Again, the same bound applies to all the four symmetries of $\cC_2$, 
showing that all the eight monotone juxtapositions have worst-case sorting time $O(\log n)$.

We now deal with the four classes $\cL$, $\cL^r$, $\PBT$ and $\PBT^r$. Noting that $\PBT\subseteq 
\cL^{\circ 2}$, it is enough to prove that $\wst{\PBT}{n}=O(\log^2 n)$, and invoke 
Observations~\ref{obs-comp} and~\ref{obs-r} to see that the same upper bound 
applies to the other three classes as well.

Composing a sequence $\pi$ with a permutation $\sigma\in\PBT$ corresponds to partitioning $\pi$ 
into an even number of blocks $B_1B_2B_3B_4\dotsb B_{2k-1}B_{2k}$, each block being a (possibly 
empty) subsequence of consecutive entries of~$\pi$, and then exchanging the position of
$B_{2j-1}$ and $B_{2j}$ for each $j\in[k]$. For instance, the permutation 
$\pi=(3,8,1,6,2,5,9,4,7)\in\cS_9$ can be partitioned as $B_1=(3,8)$, $B_2=(1)$, $B_3=\emptyset$, 
$B_4=(6,2)$, $B_5=(5,9,4)$ and $B_6=(7)$, which after exchanging adjacent odd and even blocks 
yields the permutation $B_2B_1B_4B_3B_6B_5=(1,3,8,6,2,7,5,9,4)$, which is equal to $\sigma\circ\pi$ 
for $\sigma=(3,1,2,4,5,9,6,7,8)\in\PBT_9$.

We now describe a strategy to sort an arbitrary $\pi\in\cS_n$ by $O(\log^2 n)$ steps of the 
form described above. We say that an entry $\pi(i)$ is \emph{small} if $\pi(i)\le n/2$, otherwise 
it is \emph{large}. We will use a divide-and-conquer approach, where in the initial phase, which 
will take $O(\log n)$ steps, we move all the small entries to the left of all the large ones, and 
then proceed recursively to separately and concurrently sort the small entries and the large ones. 

In the initial phase, in every step we first partition the current input sequence $\pi'$ into small 
runs and large runs, where a \emph{small run} is a maximal sequence of consecutive small values, and 
a \emph{large run} is defined similarly. Suppose $\pi'$ starts with a small element, the other case 
being analogous. It is then partitioned as $S_1L_1S_2L_2\dotsb S_kL_k$, where each $S_j$ is a small 
run and $L_j$ a large run, with $L_k$ possibly empty. Then in a single sorting step, for every even 
$j\in[k]$ we exchange the positions of the $L_{j-1}$ and $S_j$, leaving all the other runs 
unaffected. In this manner, $S_j$ becomes adjacent to $S_{j-1}$ and $L_j$ (if it exists) to 
$L_{j-1}$. Thus, every small and large run will merge with another run of the same type, except 
perhaps the rightmost small run and the rightmost large run. In $O(\log n)$ steps, we reach a 
permutation with only a single small run and a single large run, which can be swapped to ensure 
that the small run is to the left of the large one. 

We then recursively apply the same strategy to concurrently sort the $\lfloor n/2\rfloor$ small 
elements and the $\lceil n/2\rceil$ large ones. In $O(\log^2 n)$ steps, the entire sequence is 
sorted.
\end{proof}

Propositions~\ref{pro-lin1} and \ref{pro-lin2} together prove Theorem~\ref{thm-lin}.

We remark that sorting by layered permutations corresponds to sorting by a sequence of pop-stacks 
in ``genuine series'', which has been considered (for two pop-stacks) by Atkinson and 
Stitt~\cite{AS}. 

We also note that sorting by layered permutations is related to the setting of sorting by 
length-weighted reversals with linear weights, considered by Bender et al.~\cite{Bender}. More 
precisely, a single step of sorting by layered permutations can be simulated by a 
sequence of length-weighted reversals of total cost $O(n)$ in the linear-weight setting 
of~\cite{Bender}, and in particular, an improvement to the $O(\log^2 n)$ upper bound for
$\wst{\cL}{n}$ would also improve the bounds in~\cite{Bender}.

% 
% 
% 
% Our argument is based on an idea of Berendsohn~\cite{BerendsohnMs}, which implicitly relates sorting 
% time with bounds on the tree-width of a certain staircase-shaped permutation. For an integer $k$, let 
% the \emph{$(k,\cC)$-staircase} be the generalized grid class whose gridding matrix has size $k\times 
% k$, whose entries on the main diagonal are equal to $\Av(21)$, the entries on the diagonal 
% immediately below are all equal to $\cC$, and all the other entries are empty. Building upon 
% the ideas of Berendsohn~\cite{BerendsohnMs}, we can show the following fact.
% 
% \begin{proposition}\label{pro-stair}
% Let $\cC$ be a permutation class. If $\pi\in\cC^{\circ k}$ for some $\pi$ and $k$, then the 
% $(k+1,\cC)$-staircase contains a permutation $\tau$ of size $O(kn)$ whose adjacency graph 
% $G(\tau)$ contains $G(\pi)$ as a minor, and in particular $\tw(\tau)\ge \tw(\pi)$.
% \end{proposition}
% 
% We combine this proposition with the following upper bound on treewidth of graphs in a 
% $(k,\cC)$-staircase class.

\section{From \texorpdfstring{$\Omega(\log n)$}{Ω(log n)} to 1}\label{sec-all}

A simple counting argument, combined with the famous Marcus--Tardos theorem~\cite{MT} shows that any 
proper permutation class $\cC$, i.e., a class $\cC$ not containing all permutations, has at least 
logarithmic worst-case sorting time. 

\begin{proposition}
If $\cC$ is a permutation class that does not contain all permutations, then 
$\wst{\cC}{n}=\Omega(\log n)$. 
\end{proposition}
\begin{proof}
By the Marcus--Tardos theorem, there is a constant $c>0$ such that $|\cC_n|\le c^n$ for all $n$. 
Consequently, $|\cC_n^{\circ k}| \le c^{kn}$ for any $k\ge 1$.  Taking $k:=\wst{\cC}{n}$ yields
\[
n!=|S_n|=\left|\bigcup_{m=1}^k \cC_n^{\circ m}\right|\le \sum_{m=1}^k |\cC_n^{\circ m}|\le kc^{kn},
\]
which implies $k=\Omega(\log n)$.
\end{proof}

Obviously, if $\cC$ is the class of all permutations, then $\wst{\cC}{n}=1$ for all $n\ge 2$. This 
completes the proof of Theorem~\ref{thm-main}.

\section{Conclusion and open problems}\label{sec-open}

The two main open problems concern the two levels of our hierarchy where we do not have matching 
upper and lower bounds. Recall that the set $\bbX$ contains the eight monotone juxtapositions, the 
class $\cL$ of layered permutations, its reverse $\cL^r$, the class $\PBT$ of parallel block 
transpositions, and its reversal $\PBT^r$. We know that for any monotone juxtaposition, the 
worst-case sorting time is $\Theta(\log n)$, but for $\cL$, $\PBT$ and their symmetries, we only 
have a lower bound of order $\Omega(\log n)$ and an upper bound of order $O(\log^2 n)$. 

Note that Observations~\ref{obs-comp} and \ref{obs-r} imply that
$\wst{\cL}{n} = \Theta(\wst{\cL^r}{n})$, $\wst{\PBT}{n} = \Theta(\wst{\PBT^r}{n})$, and 
$\wst{\cL}{n}=O(\wst{\PBT}{n})$. We can therefore restrict our attention to $\cL$ and~$\PBT$.

\begin{problem}
What is the worst-case sorting time of the class $\cL$ of layered permutations, and of the class 
$\PBT$ of parallel block transpositions?
\end{problem}

Another gap in our bounds concerns the broad family of $\bbX$-avoiding classes. We proved that any 
$\bbX$-avoiding class has worst-case sorting time of order $\Omega(\sqrt{n})$. However, every
example of such a class that we know of has worst-case sorting time of order $\Omega(n)$, 
suggesting that the lower bound can be further improved. 

\begin{problem}
 Is there an $\bbX$-avoiding class $\cC$ with $\wst{\cC}{n}=o(n)$?
\end{problem}

\bibliography{biblio}

\begin{thebibliography}{10}

\bibitem{Ahal2008}
S.~Ahal and Y.~Rabinovich.
\newblock On complexity of the subpattern problem.
\newblock {\em SIAM Journal on Discrete Mathematics}, 22(2):629--649, 2008.
\newblock \href {https://doi.org/10.1137/S0895480104444776}
  {\path{doi:10.1137/S0895480104444776}}.

\bibitem{AABRV}
M.~H. Albert, M.~D. Atkinson, M.~Bouvel, N.~Ru\v{s}kuc, and V.~Vatter.
\newblock Geometric grid classes of permutations.
\newblock {\em Trans. Amer. Math. Soc.}, 365(11):5859--5881, 2013.
\newblock \href {https://doi.org/10.1090/S0002-9947-2013-05804-7}
  {\path{doi:10.1090/S0002-9947-2013-05804-7}}.

\bibitem{containers}
M.~H. Albert, C.~Homberger, J.~Pantone, N.~Shar, and V.~Vatter.
\newblock Generating permutations with restricted containers.
\newblock {\em Journal of Combinatorial Theory, Series A}, 157:205--232, 2018.
\newblock \href {https://doi.org/10.1016/j.jcta.2018.02.006}
  {\path{doi:10.1016/j.jcta.2018.02.006}}.

\bibitem{AB}
M.~D. Atkinson and R.~Beals.
\newblock Permutation involvement and groups.
\newblock {\em Q. J. Math.}, 52(4):415--421, 2001.
\newblock \href {https://doi.org/10.1093/qjmath/52.4.415}
  {\path{doi:10.1093/qjmath/52.4.415}}.

\bibitem{AS}
M.~D. Atkinson and T.~Stitt.
\newblock Restricted permutations and the wreath product.
\newblock {\em Discrete Math.}, 259(1-3):19--36, 2002.
\newblock \href {https://doi.org/10.1016/S0012-365X(02)00443-0}
  {\path{doi:10.1016/S0012-365X(02)00443-0}}.

\bibitem{Babai}
L.~Babai and Ákos Seress.
\newblock On the diameter of {C}ayley graphs of the symmetric group.
\newblock {\em Journal of Combinatorial Theory, Series A}, 49(1):175--179,
  1988.
\newblock \href {https://doi.org/10.1016/0097-3165(88)90033-7}
  {\path{doi:10.1016/0097-3165(88)90033-7}}.

\bibitem{Bafna}
V.~Bafna and P.~A. Pevzner.
\newblock Genome rearrangements and sorting by reversals.
\newblock In {\em 34th {A}nnual {S}ymposium on {F}oundations of {C}omputer
  {S}cience ({P}alo {A}lto, {CA}, 1993)}, pages 148--157. IEEE Comput. Soc.
  Press, Los Alamitos, CA, 1993.
\newblock \href {https://doi.org/10.1109/SFCS.1993.366872}
  {\path{doi:10.1109/SFCS.1993.366872}}.

\bibitem{Bender}
M.~A. Bender, D.~Ge, S.~He, H.~Hu, R.~Y. Pinter, S.~Skiena, and F.~Swidan.
\newblock Improved bounds on sorting by length-weighted reversals.
\newblock {\em J. Comput. System Sci.}, 74(5):744--774, 2008.
\newblock \href {https://doi.org/10.1016/j.jcss.2007.08.008}
  {\path{doi:10.1016/j.jcss.2007.08.008}}.

\bibitem{Bhatia}
S.~Bhatia, P.~Feij\~{a}o, and A.~R. Francis.
\newblock Position and content paradigms in genome rearrangements: the wild and
  crazy world of permutations in genomics.
\newblock {\em Bull. Math. Biol.}, 80(12):3227--3246, 2018.
\newblock \href {https://doi.org/10.1007/s11538-018-0514-3}
  {\path{doi:10.1007/s11538-018-0514-3}}.

\bibitem{Bona}
M.~B\'{o}na.
\newblock A survey of stack-sorting disciplines.
\newblock {\em Electron. J. Combin.}, 9(2):Article 1, 16, 2002/03.
\newblock Permutation patterns (Otago, 2003).
\newblock \href {https://doi.org/10.37236/1693} {\path{doi:10.37236/1693}}.

\bibitem{BFR}
L.~Bulteau, G.~Fertin, and I.~Rusu.
\newblock Pancake flipping is hard.
\newblock {\em J. Comput. System Sci.}, 81(8):1556--1574, 2015.
\newblock \href {https://doi.org/10.1016/j.jcss.2015.02.003}
  {\path{doi:10.1016/j.jcss.2015.02.003}}.

\bibitem{Cerbai2}
G.~Cerbai, A.~Claesson, and L.~Ferrari.
\newblock Stack sorting with restricted stacks.
\newblock {\em J. Combin. Theory Ser. A}, 173:105230, 19, 2020.
\newblock \href {https://doi.org/10.1016/j.jcta.2020.105230}
  {\path{doi:10.1016/j.jcta.2020.105230}}.

\bibitem{Cerbai1}
G.~Cerbai, A.~Claesson, L.~Ferrari, and E.~Steingr\'{\i}msson.
\newblock Sorting with pattern-avoiding stacks: the 132-machine.
\newblock {\em Electron. J. Combin.}, 27(3):Paper No. 3.32, 27, 2020.
\newblock \href {https://doi.org/10.37236/9642} {\path{doi:10.37236/9642}}.

\bibitem{CF}
G.~Cerbai and L.~Ferrari.
\newblock Permutation patterns in genome rearrangement problems: The reversal
  model.
\newblock {\em Discrete Applied Mathematics}, 279:34--48, 2020.
\newblock \href {https://doi.org/10.1016/j.dam.2019.10.012}
  {\path{doi:10.1016/j.dam.2019.10.012}}.

\bibitem{DEW}
V.~Dujmovi\'{c}, D.~Eppstein, and D.~R. Wood.
\newblock Structure of graphs with locally restricted crossings.
\newblock {\em SIAM J. Discrete Math.}, 31(2):805--824, 2017.
\newblock \href {https://doi.org/10.1137/16M1062879}
  {\path{doi:10.1137/16M1062879}}.

\bibitem{DvorakNorin}
Z.~Dvořák and S.~Norin.
\newblock Treewidth of graphs with balanced separations.
\newblock {\em Journal of Combinatorial Theory, Series B}, 137:137--144, 2019.
\newblock \href {https://doi.org/10.1016/j.jctb.2018.12.007}
  {\path{doi:10.1016/j.jctb.2018.12.007}}.

\bibitem{EPG}
A.~{Elvey Price} and A.~J. Guttmann.
\newblock Permutations sortable by two stacks in series.
\newblock {\em Advances in Applied Mathematics}, 83:81--96, 2017.
\newblock \href {https://doi.org/10.1016/j.aam.2016.09.003}
  {\path{doi:10.1016/j.aam.2016.09.003}}.

\bibitem{Fertin}
G.~Fertin, A.~Labarre, I.~Rusu, E.~Tannier, and S.~Vialette.
\newblock {\em Combinatorics of genome rearrangements}.
\newblock Computational Molecular Biology. MIT Press, Cambridge, MA, 2009.
\newblock \href {https://doi.org/10.7551/mitpress/9780262062824.001.0001}
  {\path{doi:10.7551/mitpress/9780262062824.001.0001}}.

\bibitem{Gates}
W.~H. Gates and C.~H. Papadimitriou.
\newblock Bounds for sorting by prefix reversal.
\newblock {\em Discrete Math.}, 27(1):47--57, 1979.
\newblock \href {https://doi.org/10.1016/0012-365X(79)90068-2}
  {\path{doi:10.1016/0012-365X(79)90068-2}}.

\bibitem{GHT}
J.~R. Gilbert, J.~P. Hutchinson, and R.~E. Tarjan.
\newblock A separator theorem for graphs of bounded genus.
\newblock {\em Journal of Algorithms}, 5(3):391--407, 1984.
\newblock \href {https://doi.org/10.1016/0196-6774(84)90019-1}
  {\path{doi:10.1016/0196-6774(84)90019-1}}.

\bibitem{Habermann}
N.~Habermann.
\newblock Parallel neighbor-sort (or the glory of the induction principle).
\newblock {\em Technical Report}, 1972.

\bibitem{HomVat}
C.~Homberger and V.~Vatter.
\newblock On the effective and automatic enumeration of polynomial permutation
  classes.
\newblock {\em J. Symbolic Comput.}, 76:84--96, 2016.
\newblock \href {https://doi.org/10.1016/j.jsc.2015.11.019}
  {\path{doi:10.1016/j.jsc.2015.11.019}}.

\bibitem{HV}
S.~Huczynska and V.~Vatter.
\newblock Grid classes and the {F}ibonacci dichotomy for restricted
  permutations.
\newblock {\em Electron. J. Combin.}, 13(1):Research Paper 54, 14, 2006.
\newblock \href {https://doi.org/10.37236/1080} {\path{doi:10.37236/1080}}.

\bibitem{Kececioglu}
J.~Kececioglu and D.~Sankoff.
\newblock Exact and approximation algorithms for sorting by reversals, with
  application to genome rearrangement.
\newblock {\em Algorithmica}, 13(1-2):180--210, 1995.
\newblock \href {https://doi.org/10.1007/BF01188586}
  {\path{doi:10.1007/BF01188586}}.

\bibitem{Knuth68}
D.~E. Knuth.
\newblock {\em The Art of Computer Programming, Volume {I:} Fundamental
  Algorithms}.
\newblock Addison-Wesley, 1968.

\bibitem{MT}
A.~Marcus and G.~Tardos.
\newblock Excluded permutation matrices and the {S}tanley-{W}ilf conjecture.
\newblock {\em J. Combin. Theory Ser. A}, 107(1):153--160, 2004.
\newblock \href {https://doi.org/10.1016/j.jcta.2004.04.002}
  {\path{doi:10.1016/j.jcta.2004.04.002}}.

\bibitem{Tarjan}
R.~Tarjan.
\newblock Sorting using networks of queues and stacks.
\newblock {\em J. Assoc. Comput. Mach.}, 19:341--346, 1972.
\newblock \href {https://doi.org/10.1145/321694.321704}
  {\path{doi:10.1145/321694.321704}}.

\bibitem{VatterSmall}
V.~Vatter.
\newblock Small permutation classes.
\newblock {\em Proc. Lond. Math. Soc. (3)}, 103(5):879--921, 2011.
\newblock \href {https://doi.org/10.1112/plms/pdr017}
  {\path{doi:10.1112/plms/pdr017}}.

\bibitem{VatterEH}
V.~Vatter.
\newblock An {E}rd{\H o}s-{H}ajnal analogue for permutation classes.
\newblock {\em Discrete Math. Theor. Comput. Sci.}, 18(2):Paper No. 4, 5, 2016.
\newblock \href {https://doi.org/10.46298/dmtcs.1328}
  {\path{doi:10.46298/dmtcs.1328}}.

\end{thebibliography}

\end{document}